\documentclass[letterpaper, 10 pt, conference]{ieeeconf}  % Comment this line out if you need a4paper

\IEEEoverridecommandlockouts                              % This command is only needed if 
\usepackage[utf8]{inputenc} % allow utf-8 input
\usepackage[T1]{fontenc}    % use 8-bit T1 fonts
\usepackage{dsfont}

\usepackage[maxfloats=100]{morefloats}[2015/07/22]% v1.0h
\usepackage{multirow}

\usepackage{booktabs}       % professional-quality tables
\usepackage{amsfonts}       % blackboard math symbols
\usepackage{nicefrac}       % compact symbols for 1/2, etc.
\usepackage{microtype}      % microtypography

\usepackage{enumerate}
\usepackage{lipsum}
\usepackage{mathtools}
\usepackage{cuted}
\usepackage{tabu}
\usepackage{float}
\usepackage[dvipsnames,table,xcdraw]{xcolor}
\usepackage{bbm}
\usepackage{varioref}
\usepackage{hyperref}
\usepackage{amssymb} 
\usepackage{rotating}
\usepackage{graphicx}
\usepackage{subcaption}
\usepackage[normalem]{ulem}

\usepackage{amsthm}

\DeclareUnicodeCharacter{00A0}{~}

\usepackage{algorithm}
\usepackage{algpseudocode}
\algnewcommand{\Initialize}[1]{%
	\State \textbf{Initialization:}
	\Statex {\raggedright #1}
}

\newtheorem{assumption}{Assumption}
\newtheorem{theorem}{Theorem}
\newtheorem{lemma}{Lemma}
\newtheorem{proposition}{Proposition}
\newtheorem{definition}{Definition}

\theoremstyle{plain}
\newtheorem{remark}{Remark}
\newcommand{\hdk}[1]{{\color{black}#1}}

\title{\LARGE \bf
An Incremental Gradient Method for Large-scale Distributed Nonlinearly Constrained Optimization }

\author{Harshal D. Kaushik$^1$ and Farzad Yousefian$^2$% <-this % stops a space
\thanks{*Farzad Yousefian gratefully acknowledges the support of the NSF through CAREER grant ECCS$-1944500$\hdk{.}}% <-this % stops a space
\thanks{$^1$Doctoral Candidate in the School of Industrial Engineering \& Management, Oklahoma State University, Stillwater, OK 74074, USA  {\tt\small harshal.kaushik@okstate.edu}}
\thanks{$^2$Assistant Professor in the School of Industrial Engineering \& Management, Oklahoma State University, Stillwater, OK 74074, USA  {\tt\small  farzad.yousefian@okstate.edu}}
}

\begin{document}

\maketitle
\thispagestyle{empty}
\pagestyle{empty}

%%%%%%%%%%%%%%%%%%%%%%%%%%%%%%%%%%%%%%%%%%%%%%%%%%%%%%%%%%%%%%%%%%%%%%%%%%%%%%%%
\begin{abstract}
Motivated by applications arising from sensor networks and machine learning, we consider the problem of minimizing a finite sum of nondifferentiable convex functions where each component function is associated with an agent and a hard-to-project constraint set. Among well-known avenues to address finite sum problems is the class of incremental gradient (IG) methods where a single component function is selected at each iteration in a cyclic or randomized manner. When the problem is constrained, the existing IG schemes (including projected IG, proximal IAG, and SAGA) require a projection step onto the feasible set at each iteration. Consequently, the performance of these schemes is afflicted with costly projections when the problem includes: (1) nonlinear constraints, or (2) a large number of linear constraints. Our focus in this paper lies in addressing both of these challenges. We develop an algorithm called averaged iteratively regularized incremental gradient (aIR-IG) that does not involve any hard-to-project computation. Under mild assumptions, we derive non-asymptotic  rates of convergence for both suboptimality and infeasibility metrics. Numerically, we show that the proposed scheme outperforms the standard \hdk{projected} IG methods on  distributed soft-margin support vector machine problems.
\end{abstract}

%%%%%%%%%%%%%%%%%%%%%%%%%%%%%%%%%%%%%%%%%%%%%%%%%%%%%%%%%%%%%%%%%%%%%%%%%%%%%%%%
\section{Introduction}

\begin{table*}[h]
	%	\scriptsize
	%\begin{tabu}  { | X[c] | X[c] | X[c] | X[c] | X[c] | X[c] | X[c] |}
	\vspace{0.2cm}
	\caption{{Comparison of {incremental gradient} schemes for solving finite sum problems.} }
	\centering
	\begin{tabular}{ l  l l  l l  l }
		%			\specialrule{.1em}{.05em}{.05em} 
		\toprule			
		\begin{minipage}{0.4cm} Reference \end{minipage} & Scheme & Problem class & Problem formulation & Convergence rate(s) &  Memory (per iter.) \\ 
		%			\hline
		\cmidrule(r){1-6}
		\begin{minipage}{0.4cm} { \centering \cite{nedichThesis} }\end{minipage} &  Projected IG   &    {$C_0^0$}   &   {$\min_{x\in X} \ \textstyle\sum_{i = 1}^m f_i(x)$}  & $\mathcal{O}\left(\tfrac{1}{\sqrt{k}}\right)$ &$\mathcal{O}(n)$  \\
		%			\hline 
		  \cite{BlattHeroGauchman2007, GurbuzbalanOzdaglarParrilo2017}    & IAG &    $C_{\mu,L}^{1,1}$   &   {$\min_{x\in \mathbb{R}^n} \ \textstyle\sum_{i = 1}^m f_i(x)$}  &linear & $\mathcal{O}(mn)$\\
		%			\hline
		  \cite{DefazioBachJulien2014}    & SAGA& $C_{0,L}^{1,1}$ , $C_{\mu,L}^{1,1}$   &   {$\min_{x\in X} \ \textstyle\sum_{i = 1}^m f_i(x)$}   & ${\cal O} \left(\tfrac{1}{{k}}\right)$, linear & $\mathcal{O}(mn)$ \\
		%			\hline
		 \cite{VanliGurbuzbalanOzdaglar2018}  &  Proximal IAG  &  $C_{\mu,L}^{1,1}$  &   {$\min_{x\in X} \ \textstyle\sum_{i = 1}^m f_i(x)$}  & linear &$\mathcal{O}(mn)$ \\
		%			\hline
		  \cite{GurbuzbalanOzdaglarParrilo2019}   & \begin{minipage}{2cm} IG\end{minipage} &   $C_{0,L}^{2,1}, C_{\mu,L}^{2,1}$   & \begin{minipage}{2cm} \centering {$\min_{x\in \mathbb{R}^n} \ \textstyle\sum_{i = 1}^m f_i(x)$} \end{minipage} & ${\cal O} \left(\tfrac{1}{\sqrt{k}}\right)$,   ${\cal O} \left(\tfrac{1}{{k}}\right)$  & $\mathcal{O}(n)$\\
		%			\hline
		\cmidrule(r){1-6}
		%			\hline
		\begin{minipage} {0.9cm} \textbf{This}\\ \textbf{work} \end{minipage}  & \begin{minipage}{2.5cm}   aIR-IG \end{minipage} & \begin{minipage}{2.2cm}{{$C_{0}^{0}$}} \end{minipage} & \begin{minipage}{3.6cm}   {$\min_{x\in X} \ \textstyle\sum_{i = 1}^m f_i(x)$ \\  $\ \ \ h_i(x)\leq0 \ \ \forall i\in [m]$\\  $ A_ix = b_i \ \ \forall i\in [m]$\\  $x^{(j)}\geq 0\ \  \forall j \in J$ } \end{minipage} & \begin{minipage}{3.5cm}  {suboptimality: ${\cal O} \left( k^{-0.5+b}\right)$ \\ infeasibility: ${\cal O} \left(k^{-b}\right)$ \\ for an arbitrary $0 < b < 0.5$}\end{minipage}  &$\mathcal{O}(n)$ \\
		%			\hline
		%\specialrule{.1em}{.05em}{.05em} 
		\bottomrule
	\end{tabular}
	\label{table:schemes_literature}
	\vspace*{-0.40cm}
\end{table*}

We consider a finite sum minimization subject to nonlinear inequality and linear equality functional constraints as follows:
\begin{alignat}{2}\label{initial_problem}\tag{$P$}
\underset{x\in \mathbb{R}^n }{\text{minimize }} \quad f(x) & \triangleq \textstyle\sum_{i=1}^m f_i(x) \\
 \text{subject to} \quad h_i(x) &\leq 0 \qquad && \hbox{for all } i \in \{1,\ldots,m\},\nonumber\\
  A_i x & = b_i && \hbox{for all } i \in \{1,\ldots,m\},\nonumber\\
x^{(j)}& \geq 0 && {\text{for all } j \in J \nonumber},\\
  x & \in X, \nonumber
\end{alignat}
where the component functions $f_i:\mathbb{R}^n \to \mathbb{R}$ and $h_i:\mathbb{R}^n \to \mathbb{R}$ are nonsmooth convex, $A_i \in \mathbb{R}^{d_i\times n} $, and $b_i \in \mathbb{R}^{d_i}$, for all $i \in \{1,\dots,m\}$. Also, $X\subseteq\mathbb{R}^n$ is an easy-to-project convex set and $J \subseteq \{1,\dots, n  \}$. The information about $f_i$, $h_i$, $A_i$, and $b_i$ is only known by agent $i$, while the sets $X$ and $J$ are known by all the agents. Parameters $n$, $m$, and $p\triangleq \textstyle\sum_{i=1}^md_i$ are possibly large. 
 Problem \eqref{initial_problem} arises in a breadth of applications including expected loss minimization in statistical learning \cite{RouxSchmidtBach2012} where $f_i$ is associated with a data block, as well as distributed optimization in wireless sensor networks where $f_i$ represents the local performance measure of the $i^{\text{th}}$ agent \cite{Rabbat2004}. One of the popular methods in addressing finite sum problems, in particular, in the unconstrained \hdk{regime}, is the class of incremental gradient (IG) methods where utilizing the additive structure of the problem, the algorithm cycles through the data blocks and updates the local estimates of the optimal solution in a sequential manner \cite{Bertsekas16}. While the first variants of IG schemes find their roots in addressing neural networks as early as in the '80s \cite{BertsekasIG11}, the complexity analysis of these schemes has been a trending research topic in the fields of control and machine learning in the past two decades. In addressing constrained problems with easy-to-project constraint sets, the projected incremental gradient (P-IG) method and its subgradient variant were developed \cite{NedBert2001}. In the smooth case, it is described as follows: given an initial point $x_{0,1} \in X$, where $X \subseteq \mathbb{R}^n$ denotes the constraint set, for each $k \geq 1$, consider the following update rule:
\begin{align*}
&x_{k,i+1}: = \mathcal{P}_X\left(x_{k,i}-\gamma_{k}\nabla f_i\left(x_{k,i}\right)\right) \quad \text{for all }i=1,\ldots,m,\\
&x_{k+1,1}: = x_{k,m+1} \qquad \text{for all }k \geq 0,
\end{align*} 
where $\mathcal{P}$ denotes the Euclidean projection operator and is defined as $\mathcal{P}_X(z)\triangleq \text{argmin}_{x\in X}\|x-z\|_2$ and $\gamma_k>0$ is the stepsize parameter. Recently, under the assumption of strong convexity and twice continuous differentiability of the objective function, the standard IG method was proved to converge with the rate $\mathcal{O}(1/k)$ in the unconstrained case \cite{GurbuzbalanOzdaglarParrilo2019}. This is an improvement to the previously known rate of $\mathcal{O}(1/\sqrt{k})$ for the merely convex case. Accelerated variants of IG schemes with provable convergence speeds were also developed, including  the incremental aggregated gradient method (IAG) \cite{BlattHeroGauchman2007,GurbuzbalanOzdaglarParrilo2017}, SAG \cite{RouxSchmidtBach2012}, and SAGA \cite{DefazioBachJulien2014}. While addressing the merely convex case, SAGA using averaging achieves a sublinear convergence rate, assuming strong convexity and smoothness, this is improved for non-averaging variants of SAGA and  IAG to a linear rate. 

\textbf{Existing gap:} Despite the faster rates of convergence in comparison with the standard IG method, the aforementioned methods require an excessive memory of $\mathcal{O}\left(mn\right)$ which limits their applications in the large-scale settings. Another existing challenge in the implementation of these schemes lies in addressing the hard-to-project constraints. Contending with the presence of constraints, projected (and more generally proximal) variants of the aforementioned IG schemes have been developed. However, the performance of these schemes is afflicted with costly projections when the problem includes: (1) nonlinear constraints, or (2) a large number of linear constraints. In the area of distributed optimization over networks, addressing constraints has been done to a limited extent through employing duality theory, projection, or penalty methods (see \cite{ChangNedichScaglione2014,AybatHamedani2016,HamedaniAybat2019,ScutariSun2019, NedichTatarenko2020}). We also note that a celebrated variant of the dual based schemes is the alternating direction method of multipliers (ADMM) (e.g., see \cite{MakhdoumiOzdaglar2017, KhatanaSalapaka2020,AybatHamedani2019, SunSun2020,TangDaoutidis2019}). Despite the recent advancements in this area, most ADMM methods cannot address inequality constraints with a separable structure as in \eqref{initial_problem}. Also, ADMM schemes often work under the premise that the communication graph is undirected. Indeed, despite the wide-spread application of the theory of duality and Lagrangian relaxation in addressing constrained problems in centralized regimes, there have been a limited work in the area of distributed optimization that can cope with hard-to-project constraints (see \cite{Bertsekas2015,AybatHamedani2016, HamedaniAybat2019} and the references therein). Nevertheless, the problem formulation \eqref{initial_problem} is not addressed in the aforementioned articles. \hdk{Recently,  primal-dual algorithms are proposed for finite sum convex optimization problems with conic constraints \cite{AybatHamedani2016, HamedaniAybat2020}. A recent work \cite{Jalilzadeh2021} introduced primal-dual incremental gradient method for nonsmooth convex optimization \hdk{problems}.} \hdk{Moreover}, iterative regularization (IR) has been employed as a new constraint-relaxation strategy in regimes where addressing the constraints are challenging (e.g., see \cite{AminiYousefian2018, YousefianNedichShanbhag2017, YousefianNedichShanbhag2020, KaushikYousefian2021}). Our work in this paper has been motivated by the recent success of the IR approach. To this end, our goal lies in employing the IR approach to develop an IG algorithm that can address formulation \eqref{initial_problem} without requiring any hard-to-project computation.  

{\bf Main contributions.} This work enables IG methods to address large-scale nonlinearly constrained optimization problems efficiently. To highlight our contributions, we have prepared Table \ref{table:schemes_literature}.  Our main contributions are as follows:

(i) We develop an algorithm called averaged iteratively regularized incremental gradient (aIR-IG) where at each iteration, a suitably defined stepsize and a regularization parameter are updated. Importantly, the proposed algorithm does not require any hard-to-project computation (see Algorithm \ref{alg:IR-IG_avg}). Also, aIR-IG is an incremental gradient scheme in the sense that at each iteration, only the local information of $f_i$, $h_i$, $A_i$, and $b_i$ is used by agent $i$ and agents communicate through a cycle graph. 

(ii) Under mild assumptions, we derive non-asymptotic rates of convergence for both suboptimality and infeasibility metrics. This is done through a careful choice of the stepsize and the regularization parameter that are updated iteratively. Importantly, the rate analysis in this paper is done under much weaker assumptions for functions $f_i$ in comparison with standard IG methods (see Table \ref{table:schemes_literature}).

%
%
%(iii) Numerically, we show that the proposed scheme outperforms the standard IG methods, including the projected IG, proximal IAG, and SAGA, on distributed soft-margin support vector machine problems. 

{\bf Outline.} The {remainder} of the  paper is organized as {follows}. Section \ref{sec:algo_outline} introduces the algorithm outline {for} addressing  problem \eqref{initial_problem}. We also provide the main assumptions and the preliminaries required  for the convergence analysis. Section \ref{sec:convergence_analysis} includes the convergence analysis of the proposed scheme. Section \ref{sec:numerical_implementation} contains the numerical implementation where we compare the proposed algorithm with the standard IG methods.

{\bf Notation and preliminary definitions.}  {A function} $f : \mathbb{R}^n \rightarrow \mathbb{R}$ is said to be in the class $C_{\mu,L}^{k,r}$ if $f$ is $\mu$-strongly convex in  $\mathbb{R}^n$, $k$ times continuously differentiable, and its $r^{\text{th}}$ derivative is Lipschitz continuous with constant $L$. A \hdk{nondifferentiable} $\mu$-strongly convex  function $f:\mathbb{R}^n\rightarrow\mathbb{R}$ is in the class $C_{\mu}^{0}$. For any vector $x\in \mathbb{R}^n$,  we use $\|x\|$ to denote the $\ell_2$-norm \hdk{and $x^{(j)}$ is used for denoting the j$^{\text{th}}$ component of $x$}. For problem \eqref{initial_problem}, we define  matrix $A \in \mathbb{R}^{p\times n}$ as $A \triangleq \left(A_1^T, A_2^T, \dots, A_m^T\right) ^T$ and  vector $b \in \mathbb{R}^{p\times 1}$ as $b \triangleq \left(b_1^T, b_2^T, \dots, b_m^T\right) ^T$. 
To produce a diagonal matrix \hdk{in $\mathbb{R}^{n\times n}$} from vector  $x$, we use \hdk{the} notation $diag(x).$ 
%For set $X(\subseteq \mathbb{R}^n)$, we use $int(X)$ to denote the interior of set $X$. 
For a convex function $f:\mathbb{R}^n\rightarrow\mathbb{R}$ with the domain $\text{dom}(f)$ and  any $x\in\text{dom}(f)$, vector $ \tilde{\nabla}f(x) \in\mathbb{R}^n$ with \hdk{$f(x) +\tilde{\nabla}f(x)^T(y-x)\leq f(y)$} for  all $y\in \text{dom}(f)$, is called a subgradient of $f$ at $x$. We let $\partial f(x)$ denote  the subdifferential set of  \hdk{function} $f$ at $x$.  Euclidean  projection of vector $x$ onto \hdk{a closed convex} set $X$ is denoted by  $\mathcal{P}_X(x)$. We let $[m]$ abbreviate the set $\{1,\ldots, m\}$.

\section{Algorithm Outline}\label{sec:algo_outline}

In this section, we first provide the main assumptions on problem \eqref{initial_problem} and present the outline of the algorithm. Then, we present a few preliminary results to be used in the analysis.  %Then present the proposed scheme. }
\begin{assumption}[Properties of problem \eqref{initial_problem}]  \label{assum:initial_prob:ineq_constraint}  Suppose:\\
	\noindent (a) Component function $f_i:\mathbb{R}^n\rightarrow\mathbb{R}$ is merely convex and subdifferentiable with bounded subgradients for all $i \in [m]$.  
	
	\noindent (b) Function $h_i:\mathbb{R}^n\rightarrow\mathbb{R}$ is convex and  subdifferentiable with bounded subgradients for all $i \in [m]$. 
	
	\noindent (c) The set $X$ is compact and convex.
	
	\noindent (d) The feasible set of problem \eqref{initial_problem} is nonempty.
\end{assumption}
An underlying idea in development of Algorithm \ref{alg:IR-IG_avg} is to define a regularized error metric. 
\begin{definition}\label{def:infeasibility_metric_1}	\normalfont
Consider the following term for measuring infeasibility for an agent $i$: 
	\begin{align*}%\label{eq:def_phi}
	\phi_i(x) \triangleq  \tfrac{1}{2}\|A_ix-b_i\|^2 +  h_i^+(x) +  \textstyle\sum_{j\in J}\tfrac{\max \left\{-x^{(j)}, 0\right\}}{m},
	\end{align*} 
where $h_i^+(x) \triangleq \max\{0,h_i(x)\}$ for $i \in [m]$ and all $x \in \mathbb{R}^n$. 	Further, we define $\phi(x) = \textstyle\sum_{i = 1}^m\phi_i(x)$.
\end{definition}

Then, for each agent $i$, we consider a regularized metric defined as $\phi_i(x)+\eta_kf_i(x)$ at iteration $k$. This metric captures both infeasibility and objective component function of the agent. Next, we derive a subgradient to this metric. 

Let $\partial  h_i^+(x)$ denote the subdifferential set of the function $ h_i^+$ at $x$. Consider the vector $\tilde \nabla h_i^+(x)$ defined as $\tilde \nabla h_i^+(x)\triangleq h_i^+(x)\tilde \nabla h_i(x)$ where $\tilde \nabla h_i(x)$ denotes a subgradient of function $h_i$ at $x$. Then, from the definition of subgradient mapping and the definition of $h_i^+(x)$, we have that $\tilde \nabla h_i^+(x) \in \partial  h_i^+(x)$.
Next, consider the function $\tfrac{1}{m}\textstyle\sum_{j \in J}\max\left\{0,-x^{(j)}\right\}$. A subgradient to this function is the vector $\tfrac{\mathds{1}^-(x)}{m}$ where $\mathds{1}^-(x) $ is defined a column vector $\in \mathbb{R}^n$ and the value of any component $i \in \{ 1, \dots, n \}$ is $-1$ when $x^{(i)} < 0$ and $i \in J$, otherwise that component is $0$. Let $x_{k,i} $ \hdk{in}  $ \mathbb{R}^n$ denote the iterate of agent $i$ at iteration $k$. From the above discussion, we can conclude that the subgradient of the regularized error metric for agent $i$, is given as follows:
$${ A_i^T\left(A_ix_{k,i}-b_{i}\right)}   +\tilde{\nabla} h_i^+(x_{k,i}) +\tfrac{ \mathds{1}^-\left(x_{k,i}\right)}{m}+ {\eta_{k}} \tilde{\nabla} f_{i}\left( x_{k,i}\right).$$

We are now ready to present the outline of aIR-IG scheme presented by Algorithm \ref{alg:IR-IG_avg}. At each iteration, agents update their iterates in a cyclic manner by employing the aforementioned subgradient. Each agent uses its local information including subgradients of functions $f_i$, $h_i$, as well as matrix $A_i$ and vector $b_i$. Here $\gamma_{k}$    and $\eta_k$ are the stepsize and regularization parameters, respectively. These parameters are updated at each iteration. This, indeed, is important because the convergence and rate analysis mainly depend on the choice of $\gamma_{k}$ and $\eta_{{k}}$. The key research question lies in finding suitable update rules for the two sequences so that we can achieve convergence and rate results. For the rate analysis, we employ averaging which is characterized by stepsize $\gamma_{k}$ and a scalar $0\leq r<1$.

\begin{algorithm} 
	\caption{Averaged Iteratively Regularized Incremental Gradient (aIR-IG)}\label{alg:IR-IG_avg}
	\begin{algorithmic}[1]
		\State{{\bf Input}: $x_0 \in \mathbb{R}^n$, $\bar{x}_0 := x_0$, $S_0 := \gamma_{0}^r,$ and $0\leq r <1.$} 
		\For{$k = 0,1, \dots, {N-1}$}
		\State Let {$x_{k,1} := x_k$ and} select $\gamma_k>0,$ $\eta_{k}>0$ 
		\For{$i = 1, \dots, {m}$}
		\begin{align*}
		\hspace{0.5cm}x_{k,{i+1}} & :=  \mathcal{P}_X \left({x_{k,i}}-\gamma_k\left({ A_i^T\left(A_ix_{k,i}-b_{i}\right)}  \right.\right.\\ &\left.\left. \ +\tilde{\nabla} h_i^+(x_{k,i}) +\tfrac{ \mathds{1}^-\left(x_{k,i}\right)}{m}+ {\eta_{k}} \tilde{\nabla} f_{i}\left( x_{k,i}\right)\right)\right)
		\end{align*}
		\EndFor
		\State Set $x_{k+1}\triangleq x_{k,{m+1}}$.
		
		\State Update {the weighted} average iterate as 
		\begin{align*}
		\ \ \bar{x}_{k+1}:=\tfrac{S_k\bar{x}_k  + \gamma_{k+1}^r x_{k+1}}{S_{k+1}}, \ \text{where} \ S_{k+1}:=S_k+\gamma_{k+1}^r.    &&
		\end{align*}
		\EndFor
		\State{{\bf return}: $\bar{x}_N$. }
	\end{algorithmic}
\end{algorithm}

In the following, we claim the boundedness of the subgradients $\tilde \nabla \phi_i(x)$ and $\tilde \nabla f_i(x)$ which will be used in the rate analysis in the next section.

\begin{remark}\label{rem: initial_boundednes_subgrad}
	\normalfont
	Under Assumption \ref{assum:initial_prob:ineq_constraint}, from compactness of the set $X$, the term $A_i^T\left(A_ix-b_{i}\right)$ is bounded. Also, from \hdk{the}  boundedness of subgradients of function $h_i$ and continuity of the function $h_i$ that is implied from convexity of $h_i$, we can claim that the subgradient $\tilde \nabla h_i^+(x)\triangleq h_i^+(x)\tilde \nabla h_i(x)$ is bounded on the set $X$. Consequently, we have that $\tilde \nabla \phi_i(x) \triangleq { A_i^T\left(A_ix-b_{i}\right)}   +\tilde{\nabla} h_i^+(x) +\tfrac{ \mathds{1}^-\left(x\right)}{m} $ is a bounded subgradient of $\phi_i$ for all $x \in X$. This implies that there exists a scalar $C>0$ such that for all $x \in X$, we have: $$
	\textstyle\sum_{i=1}^m \tilde{\nabla}\phi_i(x) \leq {C} \ \text{ and } \ \tilde{\nabla}\phi_i(x) \leq \tfrac{C}{m} \hbox{ for all }i \in [m].
$$
	
\end{remark}

\begin{remark}\label{rem:boundedness_subgradient} 
	\normalfont	From Assumption \ref{assum:initial_prob:ineq_constraint}, taking into account the subdifferentiability and boundedness of subgradient of \hdk{function} $f_i$, there exists a scalar $C_f>0$ such that $\text{for all } x\in X, $  
	\begin{align*}
	\textstyle\sum_{i = 1}^m\left\| \tilde{\nabla}f_i\left( x\right)\right\|\leq C_f  \text{ and } \left\|\tilde{\nabla}{ f_{i}\left( x\right)}\right\|\leq \tfrac{C_f}{m}\hbox{ for all }i \in [m].
	\end{align*}
\end{remark}

\begin{remark}\label{rem:lipschitz_parameters}
	\normalfont Taking into account Assumption \ref{assum:initial_prob:ineq_constraint}, from Theorem 3.61 in \cite{Beck2017}, functions $f_i$ and $\phi_i$ are Lipschitz continuous over set $X$. Therefore for  $x,y\in X$, and \hdk{$i\in [m]$}, 
%	\begin{align*}
	$|f_i(x)-f_i(y)|\leq \tfrac{C_f}{m}\|x-y\| \ \text{ and } \  |\phi_i(x)-\phi_i(y)|\leq \tfrac{C}{m}\|x-y\|.$
%	\end{align*}
\end{remark}

Next, we show that the sequence $\bar{x}_k$, employed in Algorithm \ref{alg:IR-IG_avg}, is a well-defined weighted average.
\begin{remark}\label{rem:averaging_convexity}
	\normalfont
	From Algorithm \ref{alg:IR-IG_avg}, the average of the iterate can be written as ${\bar{x}_{k+1}} = \textstyle\sum_{t=0}^{k}\lambda_{t,k}x_t,$ where $\lambda_{t,k}\triangleq \tfrac{\gamma_{t}^r}{\textstyle\sum_{j = 0}^k\gamma_j^r} \ \text{ for }\ t \in \{0,\dots,k\}$ denote the weights. This can be shown using induction on $k \geq 0$. For $k=0$, the relation holds directly due to the initialization $\bar{x}_0 := x_0$. To show the relation for $k+1$, assuming that it holds for $k$,  using the step $7$ in  Algorithm \ref{alg:IR-IG_avg},  and that $S_k:= \textstyle\sum_{j = 0}^k\gamma_j^r$, we have:
	\begin{align*}
	{\bar{x}_{k+1}} = \tfrac{S_k\bar{x}_k + \gamma_{k+1}^rx_{k+1}}{S_{k+1}} = \tfrac{\textstyle\sum_{t=0}^{k}\gamma_{t}^r{x}_t + \gamma_{k+1}^rx_{k+1}}{S_{k+1}} \\= \tfrac{\textstyle\sum_{t=0}^{k+1}\gamma_{t}^rx_t}{\textstyle\sum_{t=0}^{k+1}\gamma_{t}^r}  = \textstyle\sum_{t = 0}^{k+1}\lambda_{t,k}x_t.
\end{align*}
\end{remark}
\hdk{In this work, the average of the $m^{\text{th}}$ agent's iterate  is taken. We believe the rate results also hold for the average iterates of the other agents. This remains a future direction to analyze.}

The next result will be employed in the rate analysis.
\begin{lemma}[Lemma 2.14 in \cite{KaushikYousefian2021}]\label{lem:harmonic_series_bound} For any scalar $\alpha \in [0,  1)$ and  integer $N$ such that $N\geq 2^{\tfrac{1}{1-\alpha}}-1$, we have:
	$$\tfrac{{(N+1)}^{1-\alpha}}{2(1-\alpha)} \leq \textstyle\sum_{k = 0}^{N}(k+1)^{-\alpha}\leq \tfrac{{(N+1)}^{1-\alpha}}{1-\alpha}.$$	
\end{lemma}
%\begin{proof}
%	Given $\alpha \geq 0, $ and for all $x>0, $  $x^{-\alpha}$ is nonincreasing. Therefore, we have after a slight adjustment:
%	\begin{align*}
%	\textstyle\sum_{k=0}^{N}\tfrac{1}{(k+1)^{\alpha}} &= \textstyle\sum_{k=1}^{N+1}k^{-\alpha} \ = 1 + \textstyle\sum_{k=2}^{N+1}k^{-\alpha}  \leq\tfrac{{(N+1)}^{1-\alpha}}{1-\alpha}.
%	\end{align*}
%	Similar way, we can prove the lower bound as the following:
%	\begin{align*}
%	\textstyle\sum_{k=0}^{N} \tfrac{1}{(k+1)^{\alpha}}\geq \int_{1}^{N+1} x^{-\alpha}dx= \tfrac{{(N+1)}^{1-\alpha}-1}{1-\alpha}.
%	\end{align*}
%	Now, from the assumption $N\geq 2^{\tfrac{1}{1-\alpha}}-1$, we can write, 
%	\begin{align*}
%	\textstyle\sum_{k=0}^{N} (k+1)^{-\alpha} \geq \tfrac{{(N+1)}^{1-\alpha}-0.5{(N+1)}^{1-\alpha}}{1-\alpha}.
%	\end{align*}
%\end{proof}

\section{Convergence Analysis}\label{sec:convergence_analysis}
We begin with obtaining an error bound that will be employed later in the construction of bounds on the objective value and infeasibility metrics for Algorithm \ref{alg:IR-IG_avg}.
\begin{lemma}\label{lem:pseudo_bound_sequence}
	Let the sequence $\{{x}_k\}$ be generated by Algorithm \ref{alg:IR-IG_avg} and $\{\gamma_{k}\}$ and $\{\eta_{{k}}\}$ be nonincreasing positive sequences. Let Assumption \ref{assum:initial_prob:ineq_constraint} hold, $0\leq r <1$, and scalars $C, C_f>0$ be defined as in Remarks \ref{rem: initial_boundednes_subgrad} and \ref{rem:boundedness_subgradient}, respectively. Then, for any $y\in X$ and $k\geq 0, $ we have:
	\begin{align}\label{eq_36}
			2\gamma_{k}^r\eta_{k}&\left(f(x_{k}) -f(y) \right) \nonumber  +2\gamma_k^r\left(\phi\left(x_k\right)-\phi\left(y\right)\right)   \nonumber\\& \nonumber\leq \gamma_{k}^{r-1} \left\| {x_{k}}-y\right\|^2  -\gamma_{k}^{r-1}\left\| x_{k+1} -  y \right\|^2     \\&\quad +  \left(1+  \tfrac{1}{m}\right)\gamma_{k}^{r+1}\left(C +\eta_{{k}}C_f\right)^2.
	\end{align}
\end{lemma}
\begin{proof}
	Consider the update rule in step $4$ in Algorithm \ref{alg:IR-IG_avg}. For  iteration $k\geq0$, agent $i \in  \{1,\dots,m\}$, and $y\in X$, we have: %\fy{considering} the term $\| x_{k,i+1} - y \|^2$ for $y\in X$, we have:
	\begin{align*}
		&\left\| x_{k,{i+1}} -  y \right\|^2 :=\left\| \mathcal{P}_X\left({x_{k,i}}-\gamma_k\left({ A_i^T\left(A_ix_{k,i}-b_{i}\right)}  \right.\right.\right. \\ & \left.\left.\left.  + \tilde{\nabla}h_i^+(x_{k,i})  +\tfrac{\mathds{1}^-\left(x_{k,i}\right)}{m}   + {\eta_{k}} \tilde{\nabla} f_{i}\left( x_{k,i}\right)\right)\right)- \mathcal{P}_X (y) \right\|^2.
	\end{align*}
	Employing the non-expansiveness  of the projection operator, and recalling Definition \ref{def:infeasibility_metric_1} for $\phi_i(x)$,  we have:
	\begin{align*}
		&\left\| x_{k,{i+1}} -  y \right\|^2  \\&\leq \left\| {x_{k,i}}-\gamma_k\left(\tilde{\nabla}\phi_i\left(x_{k,i}\right)+ {\eta_{k}} \tilde{\nabla} f_{i}\left( x_{k,i}\right)\right)-  y \right\|^2\\&
		=  \left\| {x_{k,i}}-y\right\|^2  + \underbrace{\gamma_k^2\left\|\tilde{\nabla}\phi_i\left(x_{k,i}\right)+ {\eta_{k}} \tilde{\nabla} f_{i}\left( x_{k,i}\right) \right\|^2}_{\text{term 1}}   \\  & \quad-  2\gamma_{k}\left(\tilde{\nabla}\phi_i\left(x_{k,i}\right)+ {\eta_{k}} \tilde{\nabla} f_{i}\left( x_{k,i}\right)\right)^T\left(x_{k,{i}} -  y \right).
	\end{align*}
	Consider term 1. Employing the triangle inequality, taking into account the definitions of scalars $C$, and $C_f$, we obtain:
	\begin{align*}
	&	\left\| x_{k,{i+1}} -  y \right\|^2  \leq \left\| {x_{k,i}}-y\right\|^2 + \gamma_{k}^2\left(\tfrac{C+\eta_{k}C_f}{m}\right)^2   \\& \qquad \underbrace{- 2\gamma_{k}\left(\tilde{\nabla}\phi_i\left(x_{k,i}\right)+ {\eta_{k}} \tilde{\nabla} f_{i}\left( x_{k,i}\right)\right)^T\left(x_{k,{i}} -  y \right)}_{\text{term 2}}.
	\end{align*}
	Bounding term 2 by invoking the definition of subgradient and  the convexity of $\phi_i(x)$ and $f_i(x)$, we obtain: 
	\begin{align*}
		&\left\| x_{k,{i+1}} -  y \right\|^2  \leq   \left\| {x_{k,i}}-y\right\|^2 + \gamma_{k}^2\left(\tfrac{C+\eta_{k}C_f}{m}\right)^2  \\&\qquad + 2\gamma_{k}\eta_{k}\left(f_i(y) - f_i(x_{k,{i}})\right)  +2\gamma_k\left(\phi_i(y) -\phi_i\left(x_{k,i}\right)\right).
	\end{align*}
	Taking summation  over all the agents  $i\in \{ 1,\dots, m \}$, %\hdkSecond{recalling $x_{k,1}:=x_k$, and  $x_{k+1} \triangleq x_{k,m+1},$ we have:}
	\begin{align*}
		&\left\| x_{k+1} -  y \right\|^2  \leq    \left\| {x_{k}}-y\right\|^2 + 2\gamma_{k}\eta_{k}\textstyle\sum_{i=1}^m\left(f_i(y) - f_i(x_{k,{i}})\right)\\& + \gamma_{k}^2\textstyle\sum_{i=1}^m\left(\tfrac{C+\eta_{k}C_f}{m}\right)^2  +2\gamma_k\textstyle\sum_{i=1}^m\left(\phi_i(y)-\phi_i\left(x_{k,{i}}\right)\right).
	\end{align*}
	Adding and subtracting $2\gamma_{k}\textstyle\sum_{i = 1}^m\phi_i\left(x_k\right)$$+2\gamma_{k}\eta_{k}\textstyle\sum_{i = 1}^mf_i\left(x_k\right)$, and taking into account Definition \ref{def:infeasibility_metric_1}, we have:
	\begin{align*}
		& \left\| x_{k+1} -  y \right\|^2 \leq \left\| {x_{k}}-y\right\|^2 + \tfrac{\gamma_{k}^2\left(C+\eta_{k}C_f\right)^2}{m} \\& + 2\gamma_{k}\eta_{k}\left(f(y) - f(x_{k})\right) +2\gamma_k\left(\phi(y)- \phi\left(x_{k}\right)\right)\\&  +2 \gamma_{k} \textstyle\sum_{i=1}^m\left( \phi_i\left(x_k\right) - \phi_i\left(x_{k,i}\right) + \eta_{k}\left(f_i\left(x_k\right) -  f_i\left(x_{k,i}\right)\right)\right),
			\\& \leq\left\| {x_{k}}-y\right\|^2 + \tfrac{\gamma_{k}^2\left(C+\eta_{k}C_f\right)^2}{m}  \\&  + 2\gamma_{k}\eta_{k}\left(f(y) - f(x_{k})\right) +2\gamma_k\left(\phi(y)  - \phi\left(x_{k}\right)\right) \\&\hspace{-0.1cm}+ 2 \gamma_{k}\textstyle\sum_{i=1}^m \left(\underbrace{\left|\phi_i\left(x_k\right) - \phi_i\left(x_{k,i}\right)\right| }_{\text{term 3}} + \eta_{k}\underbrace{\left|f_i\left(x_k\right) -  f_i\left(x_{k,i}\right)\right|}_{\text{term 4}} \right).
	\end{align*}
	From Remark \ref{rem:lipschitz_parameters}, bounding terms 3 and 4, we have: 
	\begin{align}\label{eqn:error_bound_1}
			&\left\| x_{k+1} -y \right\|^2\leq\left\| {x_{k}}-y\right\|^2 + \tfrac{\gamma_{k}^2\left(C+\eta_{k}C_f\right)^2}{m}  \nonumber \\&\quad +2\gamma_{k}\eta_{k}\left(f(y) - f(x_{k})\right) +2\gamma_k\left(\phi(y)- \phi\left(x_{k}\right)\right) \nonumber \\ & \quad + \tfrac{2 \gamma_{k}\left(C +\eta_{k}C_f\right)}{m}\textstyle\sum_{i=2}^m \underbrace{\left\| x_k - x_{k,i} \right\|}_{\text{term 5}}.  
		\end{align}
		Note that from Algorithm \ref{alg:IR-IG_avg},  for $i=1$, we have $\|x_k - x_{k,1}\| = 0$. Consider term 5 in relation \eqref{eqn:error_bound_1}. Applying induction on $i$, we  bound  term 5 as $\| x_k - x_{k,i} \|\leq (i)\gamma_{k}\left(C+\eta_{k}C_f\right)/m$ for any $i= 2,\dots, m.$  For  $i=2$,   from Algorithm \ref{alg:IR-IG_avg}, we have:
		\begin{align*}
		&\| x_k - x_{k,2} \| = \left\| \mathcal{P}_X \left(x_{k,1}\right) \right. \\& \hspace{1cm}\left. -  \mathcal{P}_X \left( x_{k,1} - \gamma_{k}\left(\tilde{\nabla}\phi_1(x_{k,1}) + \eta_{k}\tilde{\nabla} f_1\left(x_{k,1}\right)\right)  \right)  \right\|\\  & \leq  \gamma_{k}\left\|\tilde{\nabla}\phi_1(x_{k,1}) + \eta_{k}\tilde{\nabla} f_1\left(x_{k,1}\right) \right\| \leq \gamma_{k}\left(C + \eta_k C_f\right)/m.
		\end{align*}
		Now, suppose the hypothesis statement holds for some $i \geq 2$. Then, we can write:
		\begin{align*}
			&\| x_k - x_{k,i+1} \| = \\& \left\| \mathcal{P}_X \left(x_{k}\right) - \mathcal{P}_X \left( x_{k,i} - \gamma_{k}\left(\tilde{\nabla}\phi_i(x_{k,i}) + \eta_{k}\tilde{\nabla} f_i\left(x_{k,i}\right)\right)  \right)  \right\|\\  & \leq  \left\| x_{k} - x_{k,i} \right\| +  \gamma_{k}\left\| \tilde{\nabla}\phi_i(x_{k,i}) + \eta_{k}\tilde{\nabla} f_i\left(x_{k,i}\right)  \right\|\\ & \leq  \left\| x_{k} - x_{k,i} \right\| +  \tfrac{\gamma_{k}\left(C + \eta_k C_f\right)}{m} \leq  \tfrac{(i+1)\gamma_{k}\left(C + \eta_k C_f\right)}{m}. 		
	\end{align*}
	Therefore, the hypothesis statement holds for any $i\geq 2$. Substituting the bound for term 5 in equation \eqref{eqn:error_bound_1}, we have:
		\begin{align*}
			&\left\| x_{k+1} -y \right\|^2\leq\left\| {x_{k}}-y\right\|^2 + \tfrac{\gamma_{k}^2\left(C+\eta_{k}C_f\right)^2}{m} \nonumber \\ &  \qquad\qquad\quad + 2\gamma_{k}\eta_{k}\left(f(y) - f(x_{k})\right)  +2\gamma_k\left(\phi(y)- \phi\left(x_{k}\right)\right) \nonumber \\ & \qquad\qquad\quad + \tfrac{2 \gamma_{k}\left(C +\eta_{k}C_f\right)}{m}\textstyle\sum_{i=2}^m \tfrac{(i)\gamma_{k}\left(C + \eta_k C_f\right)}{m}  \\
%			& = \left\| {x_{k}}-y\right\|^2 + \tfrac{2\gamma_{k}^2C^2}{m}+ \tfrac{2\gamma_{k}^2\eta_{{k}}^2C_f^2}{m} \nonumber \\ & + 2\gamma_{k}\eta_{k}\left(f(y) - f(x_{k})\right)  +2\gamma_k\left(\phi(y)- \phi\left(x_{k}\right)\right)\nonumber \\ &  + \gamma_{k}^2\left(1 + \tfrac{1}{m}\right)\left(C^2 + \eta_{k}^2C_f^2 + 2\eta_{k}CC_f\right)\\
%			& = \left\| {x_{k}}-y\right\|^2 + \tfrac{2\gamma_{k}^2C^2}{m}+ \tfrac{2\gamma_{k}^2\eta_{{k}}^2C_f^2}{m} \nonumber \\ &+ 2\gamma_{k}\eta_{k}\left(f(y) - f(x_{k})\right)+2\gamma_k\left(\phi(y)- \phi\left(x_{k}\right)\right)\nonumber \\ &  + \left(1 + \tfrac{1}{m}\right)\gamma_{k}^2C^2 + \eta_{k}^2\gamma_{k}^2\left(1 + \tfrac{1}{m}\right)C_f^2\\& + 2\gamma_{k}^2\eta_{k}\left(1 + \tfrac{1}{m}\right)CC_f\\
			&\qquad\quad\quad = \left\| {x_{k}}-y\right\|^2 + \left(1 + \tfrac{1}{m}\right)\gamma_{k}^2\left(C+\eta_{{k}}C_f\right)^2\\&\qquad\qquad\quad +  2\gamma_{k}\eta_{k}\left(f(y) - f(x_{k})\right) +2\gamma_k\left(\phi(y)- \phi\left(x_{k}\right) \right).
		\end{align*}
		Multiplying both sides by the positive term $\gamma_{k}^{r-1}$,
%		\begin{align*}
%			& \gamma_{k}^{r-1}\left\| x_{k+1} -  y \right\|^2  \leq   \gamma_{k}^{r-1} \left\| {x_{k}}-y\right\|^2 \\ & + \left(m^2 + m +  \tfrac{2}{m}\right)\gamma_{k}^{r+1}C^2  + \left(m^2 + m +  \tfrac{2}{m}\right)\gamma_{k}^{r+1}\eta_{{k}}^2C_f^2\\ & + 2\gamma_{k}^r\eta_{k}\left(f(y) - f(x_{k})\right) +2\gamma_k^r\textstyle\sum_{i=1}^m\phi_i(y)\\ &-2\gamma_k^r\textstyle\sum_{i=1}^m\phi_i(x_k)+ 2\gamma_{k}^{r+1}\eta_{k}\left(m^2 + m\right)CC_f.
%	\end{align*}
%	and from the definition of $\phi_i(x)$, 
	we obtain the desired result.
\end{proof}
Next we construct the error bounds for Algorithm \ref{alg:IR-IG_avg} in terms of the sequences $\{\gamma_{k}\}$ and $\{\eta_{{k}}\}$.
\begin{proposition}[Error bounds for Algorithm \ref{alg:IR-IG_avg}]\label{prop:error_bound_aIR-IG}
	Consider problem \eqref{initial_problem}. 
	Let $\bar{x}_N$  be generated by Algorithm \ref{alg:IR-IG_avg} after $N$ iterations and  $\{\gamma_{k}\}$ and $\{\eta_{k}\}$ be nonincreasing and strictly positive sequences. Further, let Assumption \ref{assum:initial_prob:ineq_constraint} hold, scalars $C_f, C >0$, and parameter $0 \leq r<1$. Let scalars $M, M_f>0$ be defined such that we have: $\|x\|\leq M \ \text{ and } \ |f(x)|\leq M_f \quad \text{for all } \ x \in X. $  Then for any optimal solution $x^*$ to \eqref{initial_problem}, we have the following: 
	\begin{align*}
			(a)\ f(\bar{x}_{N}) -f(x^*) \leq &\left( \textstyle\sum_{k = 0}^{N} \gamma_{k}^r\right)^{-1}\left(\tfrac{2M^2\gamma_{N}^{r-1}}{\eta_{N}}  \right.\\&  \left. +\left(1 +  \tfrac{1}{m}\right)\tfrac{\left(C+\eta_{{0}}C_f\right)^2}{2}\textstyle\sum_{k = 0}^{N}  \tfrac{\gamma_{k}^{r+1}}{\eta_{{k}}}\right).
	\end{align*}
	\begin{align*}
		(b)\ \phi\left(\bar{x}_N\right)
		\leq& \left(\textstyle\sum_{k = 0}^{N}\gamma_k^r\right)^{-1}\left(2M^2  \gamma_{N}^{r-1} + 2M_f\textstyle\sum_{k = 0}^{N} \gamma_{k}^r\eta_{k}  \right. \\ &\quad\qquad\left.+
		 \left(1 +\tfrac{1}{m}\right) \tfrac{\left(C+\eta_{{0}}C_f\right)^2}{2}\textstyle\sum_{k = 0}^{N}\gamma_{k}^{r+1} \right).
	\end{align*} 
\end{proposition}
\begin{proof}
	Consider relation \eqref{eq_36} from Lemma \ref{lem:pseudo_bound_sequence},  for any $y\in X$.
	%	\begin{align*}
	%	&2\gamma_{k}^r\eta_{k}\left(f(x_{k}) -f(y) \right)+2\gamma_k^r\left(\textstyle\sum_{i=1}^m\tfrac{1}{2}\|A_ix_k-b_i\|^2 \right. \\ &\left. +  \textstyle\sum_{i = 1}^mh_i^+(x_k) + \textstyle\sum_{j \in J}\max \left\{-x_k^{(j)}, 0\right\}\right)\nonumber \\ & -2\gamma_k^r\left(\textstyle\sum_{i=1}^m\tfrac{1}{2}\|A_iy-b_i\|^2+  \textstyle\sum_{i = 1}^mh_i^+(y)  \right.\\& \left.+ \textstyle\sum_{j \in J}\max \left\{-y^{(j)}, 0\right\}\right)  \nonumber\leq \gamma_{k}^{r-1} \left\| {x_{k}}-y\right\|^2  \\ & -\gamma_{k}^{r-1}\left\| x_{k+1} -  y \right\|^2    +  \tfrac{2\gamma_{k}^{r+1}C^2}{m}+\tfrac{2\gamma_{k}^{r+1}\eta_{{k}}^2C_f^2}{m} .
	%	\end{align*}
	Substituting $y$ by $x^*$ and taking into account the feasibility of the vector $x^*$ to problem \eqref{initial_problem}, we obtain:
	\begin{align*}
			 2\gamma_{k}^r\eta_{k}\left(f(x_{k}) -f(x^*) \right) \nonumber & +2\gamma_k^r\phi\left(x_k\right)  \\\hdk{\leq}& \gamma_{k}^{r-1} \left(\left\| {x_{k}}-x^*\right\|^2  -\left\| x_{k+1} -  x^* \right\|^2    \right) \\&\nonumber  +  \left(1+  \tfrac{1}{m}\right)\gamma_{k}^{r+1}\left(C +\eta_{{k}}C_f\right)^2.
		\end{align*}
%		From the nonnegativity of  $2\gamma_k^r \phi\left(x_k\right)$, we have:
%		\begin{align*}
%			&2\gamma_{k}^r\eta_{k}\left(f(x_{k}) -f(x^*) \right)\leq \gamma_{k}^{r-1}\phi \left(x_k\right) + 2\gamma_{k}^{r+1}\eta_{k}\left(m^2 \right.\\&\left.+ m\right)CC_f     +  \left(m^2 + m +  \tfrac{2}{m}\right)\left(\gamma_{k}^{r+1}C^2+\gamma_{k}^{r+1}\eta_{{k}}^2C_f^2\right).
%		\end{align*}
		Taking into account the nonnegativity of $ 2\gamma_k^r \phi\left(x_k\right) $  and dividing both sides by $2\eta_k$, we have: 
		\begin{align}\label{eq_38}
			 \gamma_{k}^r\left(f(x_{k}) -f(x^*) \right)&\leq   \tfrac{\gamma_{k}^{r-1}}{2\eta_{{k}}}\left( \left\| {x_{k}}-x^*\right\|^2  - \left\| x_{k+1} -  x^* \right\|^2\right)\nonumber \\&  +\left(1 +  \tfrac{1}{m}\right) \tfrac{\gamma_{k}^{r+1}\left(C+\eta_{{k}}C_f\right)^2}{2\eta_k}.
		\end{align}
		Adding and subtracting $\tfrac{\gamma_{k-1}^{r-1}}{2\eta_{k-1}}\|x_k-x^*\|^2$ in the above, 
		\begin{align}\label{eqn:prop_1_1}
			&\gamma_{k}^r\left(f(x_{k}) -f(x^*) \right)\leq   \tfrac{\gamma_{k-1}^{r-1}}{2\eta_{k-1}}\|x_k-x^*\|^2\nonumber \\& \quad-\tfrac{\gamma_{k}^{r-1}}{2\eta_{{k}}}\left\| x_{k+1} -  x^* \right\|^2   + \underbrace{\left(\tfrac{\gamma_{k}^{r-1}}{2\eta_{k}} - \tfrac{\gamma_{k-1}^{r-1}}{2\eta_{k-1}}\right)}_{\text{term 1}}\|x_k-x^*\|^2 \nonumber\\& \quad +\underbrace{\left(1 +  \tfrac{1}{m}\right) \tfrac{\gamma_{k}^{r+1}\left(C+\eta_{{k}}C_f\right)^2}{2\eta_k}}_{\text{term 2}}.
		\end{align}
		Recalling the definition for scalar $M$, we have: 
		\begin{align}\label{eq:bound_X}
			\|x_k-x^*\|^2 \leq 2\|x_k\|^2 + 2\|x^*\|^2  \leq 4M^2. 
		\end{align}
		Taking into account  $r<1$ and the nonincreasing property of the sequences $\{\gamma_{k}\}$ and $\{\eta_{k}\}$, we have: term 1 $\geq 0$. Bounding  term 2 in equation \eqref{eqn:prop_1_1}, we have:
		\begin{align*}%\label{eqn:prop_1_1}
&\gamma_{k}^r\left(f(x_{k}) -f(x^*) \right)\leq   \tfrac{\gamma_{k-1}^{r-1}}{2\eta_{k-1}}\|x_k-x^*\|^2\nonumber \\& \quad  -\tfrac{\gamma_{k}^{r-1}}{2\eta_{{k}}}\left\| x_{k+1} -  x^* \right\|^2  + \left(\tfrac{\gamma_{k}^{r-1}}{2\eta_{k}} - \tfrac{\gamma_{k-1}^{r-1}}{2\eta_{k-1}}\right)4M^2\nonumber \\& \quad  +\left(1 +  \tfrac{1}{m}\right) \tfrac{\left(C+\eta_0C_f\right)^2\gamma_{k}^{r+1}}{2\eta_k}.
\end{align*}
		Next, taking \hdk{summations} over $k = 1,\dots, N$,  we obtain:
		\begin{align}\label{eq_39}
			&\textstyle\sum_{k = 1}^{N} \gamma_{k}^r\left(f(x_{k}) -f(x^*) \right)\leq  \tfrac{\gamma_{0}^{r-1}}{2\eta_{0}}\|x_1-x^*\|^2\nonumber\\&-\tfrac{\gamma_{N}^{r-1}}{2\eta_{N}}\|x_{N+1}-x^*\|^2+  \left(\tfrac{\gamma_{N}^{r-1}}{2\eta_{N}} - \tfrac{\gamma_{0}^{r-1}}{2\eta_{0}}\right)4M^2\nonumber\\ &+\left(1 +  \tfrac{1}{m}\right)\tfrac{\left(C+\eta_{{0}}C_f\right)^2}{2}\textstyle\sum_{k = 1}^{N}  \tfrac{\gamma_{k}^{r+1}}{\eta_{{k}}}.
		\end{align}
		Rewriting  equation \eqref{eq_38} for $k = 0$,  we have:
		\begin{align*}
			\gamma_{0}^r\left(f(x_{0}) -f(x^*) \right)\leq  & \tfrac{\gamma_{0}^{r-1} }{2\eta_{{0}}}\left(\left\| {x_{0}}-x^*\right\|^2-  \left\| x_{1} -  x^* \right\|^2 \right)  \\& + \left(1 +  \tfrac{1}{m}\right)\tfrac{\gamma_{0}^{r+1}\left(C^2+\eta_{{0}} C_f\right)^2}{2\eta_{{0}}}.
		\end{align*}
		Adding the preceding relation with   \eqref{eq_39}, we obtain:
		\begin{align*}
			&\textstyle\sum_{k = 0}^{N} \gamma_{k}^r\left(f(x_{k}) -f(x^*) \right)\leq 2M^2\left(\tfrac{\gamma_{N}^{r-1}}{\eta_{N}}-\tfrac{\gamma_{0}^{r-1}}{\eta_{{0}}}\right) \\& -\tfrac{\gamma_{N}^{r-1}}{2\eta_{N}} \left\| x_{N+1} -  x^* \right\|^2 + \tfrac{\gamma_{0}^{r-1} \left\| {x_{0}}-x^*\right\|^2}{2\eta_{{0}}}  \\&+\left(1 +  \tfrac{1}{m}\right)\tfrac{\left(C+\eta_{{0}}C_f\right)^2}{2}\textstyle\sum_{k = 0}^{N}  \tfrac{\gamma_{k}^{r+1}}{\eta_{{k}}}.
		\end{align*}
		Further from \eqref{eq:bound_X}, and neglecting the nonpositive term, 
		\begin{align*}
			&  \textstyle\sum_{k = 0}^{N} \gamma_{k}^r\left(f(x_{k}) -f(x^*) \right) \leq {2M^2\gamma_{N}^{r-1}}/{\eta_{N}}
			\\&+\left(1 +  \tfrac{1}{m}\right)\tfrac{\left(C+\eta_{{0}}C_f\right)^2}{2}\textstyle\sum_{k = 0}^{N}  \tfrac{\gamma_{k}^{r+1}}{\eta_{{k}}}.
		\end{align*}
		Next, dividing both sides by  $\textstyle\sum_{k = 0}^{N} \gamma_{k}^r$, 
		\hdk{\begin{align*}
			& \left(\textstyle\sum_{k = 0}^{N} \gamma_{k}^r\right)^{-1} \textstyle\sum_{k = 0}^{N} \gamma_{k}^r\left(f(x_{k}) -f(x^*) \right) \leq \left(\textstyle\sum_{k = 0}^{N} \gamma_{k}^r\right)^{-1}\\&\left( {2M^2\gamma_{N}^{r-1}}/{\eta_{N}}
			+\left(1 +  \tfrac{1}{m}\right)\tfrac{\left(C+\eta_{{0}}C_f\right)^2}{2}\textstyle\sum_{k = 0}^{N}  \tfrac{\gamma_{k}^{r+1}}{\eta_{{k}}}\right).
		\end{align*}
		Taking into account} the  convexity of  $f$ and  recalling Remark \ref{rem:averaging_convexity}, we obtain the result. \\
	%	\begin{align*}
	%	f(\bar{x}_{N}) -f(x^*) \leq & \left( \textstyle\sum_{k = 0}^{N} \gamma_{k}^r\right)^{-1}\left(\tfrac{2M^2\gamma_{N}^{r-1}}{\eta_{N}}  \right. \\ &\left.+\tfrac{C^2}{m}\textstyle\sum_{k = 0}^{N}  \tfrac{\gamma_{k}^{r+1}}{\eta_{{k}}}+\tfrac{C_f^2}{m}\textstyle\sum_{k =0}^{N} {\gamma_{k}^{r+1}\eta_{{k}}}\right).
	%	\end{align*}
	(b) Consider equation \eqref{eq_36}. Writing it  for $y:=x^*\in X$,
	\begin{align*}
	&2\gamma_k^r\phi\left(x_k\right)   \nonumber\leq 2\gamma_{k}^r\eta_{k}\left(f\left(x^*\right) - f(x_{k})  \right) + \gamma_{k}^{r-1} \left(\left\| {x_{k}}-x^*\right\|^2\right.\\&\left.-\left\| x_{k+1} -  x^* \right\|^2\right)+  \left(1+  \tfrac{1}{m}\right)\gamma_{k}^{r+1}\left(C +\eta_{{k}}C_f\right)^2.
	\end{align*}
	Recalling the definition of  $M_f$, we have, $\ |f(x^*) - f(x_k)|\leq 2M_f$.   Bounding the preceding inequality, 
	\begin{align}\label{eq39}
	&	2\gamma_k^r\phi(x_k)\nonumber \leq  \gamma_{k}^{r-1} \left(\left\| {x_{k}}-x^*\right\|^2  -\left\| x_{k+1} - x^* \right\|^2\right)  \\& + 4\gamma_{k}^r\eta_{k}M_f+ \left(1+  \tfrac{1}{m}\right)\gamma_{k}^{r+1}\left(C +\eta_{{k}}C_f\right)^2.
	\end{align}
	Adding and subtracting $\gamma_{k-1}^{r-1}\|x_k-x^*\|^2$ in the above,
	\begin{align*}
		2\gamma_k^r\phi(x_k) \leq&   \gamma_{k-1}^{r-1} \left\| {x_{k}}-x^*\right\|^2-\gamma_{k}^{r-1}\left\| x_{k+1} -  x^* \right\|^2 \\& + 4\gamma_{k}^r\eta_{k}M_f     + \underbrace{\left(\gamma_{k}^{r-1} -\gamma_{k-1}^{r-1} \right) \left\| {x_{k}}-x^*\right\|^2}_{\text{term 3}} \\&+  \nonumber \underbrace{\left(1+  \tfrac{1}{m}\right)\gamma_{k}^{r+1}\left(C +\eta_{{k}}C_f\right)^2}_{\text{term 4}}.
	\end{align*}
	Using the nonincreasing property of $\{\gamma_{k}\}$ and $\{\eta_{k}\}$, recalling $0\leq r<1$, we have $ \gamma_{k}^{r-1} -\gamma_{k-1}^{r-1} > 0 $, and $\left(1 +  \tfrac{1}{m}\right) \gamma_{k}^{r+1} > 0$. Further, from the boundedness of set $X$, we have: term 3 $< \left( \gamma_{k}^{r-1} -\gamma_{k-1}^{r-1} \right) 4M^2 $, and term 4 $< \left(1 +  \tfrac{1}{m}\right) \gamma_{k}^{r+1}\left(C + \eta_{0}C_f\right)^2$. Next,  taking \hdk{summations} over $k = 1,\dots, N, $ and dropping the nonpositive terms, we obtain:
	\begin{align}\label{eq40}
		& 2\textstyle\sum_{k = 1}^{N}\gamma_k^r\phi(x_k)\leq  \gamma_{0}^{r-1} \left\| {x_{1}}-x^*\right\|^2  + 4M^2 \left(\gamma_{N}^{r-1} -\gamma_{0}^{r-1} \right)      \nonumber \\   &+ \left(1 +\tfrac{1}{m}\right) \left(C+\eta_{{0}}C_f\right)^2\textstyle\sum_{k = 1}^{N}\gamma_{k}^{r+1}+4M_f\textstyle\sum_{k = 1}^{N}\gamma_{k}^r\eta_{k}.
	\end{align}
	Writing equation \eqref{eq39} for $k=0$,  we have:
	\begin{align*}
		2\gamma_0^r\phi(x_0)  \leq  & \gamma_{0}^{r-1} \left(\left\| {x_{0}}-x^*\right\|^2   -\left\| x_{1} -  x^* \right\|^2\right) + 4\gamma_{0}^r\eta_{0}M_f \nonumber\\&  + \left(1 +\tfrac{1}{m}\right) \gamma_{0}^{r+1}\left(C+\eta_{{0}}C_f\right)^2\nonumber.
	\end{align*}
	Adding this into equation \eqref{eq40}, we have:
	\begin{align*}
		&	2\textstyle\sum_{k = 0}^{N}\gamma_k^r\phi(x_k)\leq \gamma_{0}^{r-1} \left\| {x_{0}}-x^*\right\|^2+ 4M^2 \left( \gamma_{N}^{r-1}- \gamma_{0}^{r-1} \right)   \\&+ \left(1 +\tfrac{1}{m}\right) \left(C+\eta_{{0}}C_f\right)^2\textstyle\sum_{k = 0}^{N}\gamma_{k}^{r+1}+4M_f\textstyle\sum_{k = 0}^{N}\gamma_{k}^r\eta_{k}.
	\end{align*}
	Bounding $\|x_0-x^*\|^2$ from equation \eqref{eq:bound_X}, dividing both sides by  $\textstyle\sum_{k = 0}^{N} \gamma_{k}^r$, \hdk{taking into account} the  convexity of   $\phi(x_k)$, and from Remark \ref{rem:averaging_convexity}, we obtain  the required result.
\end{proof}
Next, we present the suboptimality and infeasibility convergence rate statements for the proposed algorithm.
\begin{theorem}[Suboptimality and infeasibility rate results]\label{thm:rates}
	Consider Algorithm \ref{alg:IR-IG_avg}. Let Assumption \ref{assum:initial_prob:ineq_constraint} hold. Consider scalars $M,M_f>0$ such that $\|x\|\leq M$ and $|f(x)| \leq M_f \text{ for all } x \in X.$ Let $\bar{x}_N$ be generated by Algorithm \ref{alg:IR-IG_avg} after $N$ iterations. Let $\{\gamma_{k}\}$ and $\{\eta_{{k}}\}$ be the stepsize and regularization parameter sequences generated using $	\gamma_{k} = \tfrac{\gamma_{0}}{\sqrt{1+k}},\	\eta_{k} = \tfrac{\eta_{0}}{(1+k)^b}$, where $\gamma_{{0}}, \eta_{{0}} > 0,$ and $0< b <0.5$. Then, for any optimal solution $x^*$ to problem \eqref{initial_problem}, we have:
	\begingroup \allowdisplaybreaks
	\begin{align}\label{eqn:rate_f}
		(a) \  	& {f\left(\bar x_N\right)}-f(x^*)  \leq\tfrac{2-r}{\gamma_0^r(N+1)^{0.5-b}}\left(\tfrac{2M^2}{\eta_0\gamma_0^{1-r}}\right.\nonumber\\
		&\left.+\tfrac{\left(m+1\right)\gamma_0^{1+r}\left(C+\eta_0C_f\right)^2}{2m\eta_0(0.5-0.5r+b)}\right).\\
		(b)\  & \phi\left(\bar{x}_N\right)\nonumber 
		\leq \tfrac{(2-r)}{(N+1)^{b}} \left(\tfrac{2M^2}{\gamma_0} + \tfrac{2M_f \eta_0 }{(1-0.5r-b)} \right.\nonumber\\
		& \left. +\tfrac{\left(m+1\right)\left(C+\eta_0C_f\right)^2\gamma_0}{2m(0.5-0.5r)}   \right).\label{eqn:rate_infeas}
	\end{align}
	\endgroup
%
%\begingroup
\allowdisplaybreaks
%\begin{align*}
%	(a)\ &f(x) \\
%	&\leq 1.\\
%	(a)\ &g(x)\leq 2.
%\end{align*}
%\endgroup
\end{theorem}
\begin{proof} 
	Taking Proposition \ref{prop:error_bound_aIR-IG} (a) and (b) into account,
%	\begin{align}\label{eqn:bound_f}
%		\leq & \left( \textstyle\sum_{k = 0}^{N} \gamma_{k}^r\right)^{-1}\left(\tfrac{2M^2\gamma_{N}^{r-1}}{\eta_{N}}  +\left(m^2 + m\right)CC_f\textstyle\sum_{k = 0}^N\gamma_{k}^{r+1}\right. \nonumber\\ &\left.+ \left(m^2+m+\tfrac{2}{m}\right)\tfrac{C^2 + \eta_0^2C_f^2}{2}\textstyle\sum_{k = 0}^{N}  \tfrac{\gamma_{k}^{r+1}}{\eta_{{k}}}\right).
%	\end{align}
%	Now consider Proposition \ref{prop:error_bound_aIR-IG} (b),
%	\begin{align}\label{eqn:bound_infeas}
%		&\textstyle\sum_{i=1}^m\tfrac{1}{2}\|A_i\bar{x}_N-b_i\|^2 +  \textstyle\sum_{i = 1}^mh_i^+(\bar{x}_N)   + \textstyle\sum_{j \in J}\max \left\{-\bar{x}_N^{(j)}, 0\right\}\nonumber
%		\\		\leq & \left( \textstyle\sum_{k = 0}^{N} \gamma_{k}^r\right)^{-1}\left(2M^2\gamma_{N}^{r-1}  +\left(m^2 + m\right)CC_f\textstyle\sum_{k = 0}^N\gamma_{k}^{r+1}\eta_{k}\right. \nonumber\\ &\hspace{-0.4cm}\left.+ \left(m^2+m+\tfrac{2}{m}\right)\tfrac{C^2 + \eta_0^2C_f^2}{2}\textstyle\sum_{k = 0}^{N}  \gamma_{k}^{r+1}+ 2M_f\textstyle\sum_{k = 0}^{N} \gamma_{k}^r\eta_{k} \right).
%	\end{align}
	let us define the following terms: 
	\begin{align*}%\label{eqn:three_terms_abr2}
		& \Lambda_{N,1} \triangleq \textstyle\sum_{k=0}^{N}\gamma_k^r, \quad  \Lambda_{N,2}\triangleq  \tfrac{2M^2\gamma_{{N}}^{r-1}}{\eta_{N}} , \\ & \Lambda_{N,3} \triangleq \left(1+\tfrac{1}{m}\right)\tfrac{\left(C+\eta_0C_f\right)^2}{2}\textstyle\sum_{k=0}^{N}\eta_k^{-1}\gamma_k^{r+1}, 
\\&  \Lambda_{N,4} \triangleq 2M^2\gamma_{{N}}^{r-1}, \quad\Lambda_{N,5} \triangleq 2M_f	 \textstyle\sum_{k=0}^{N}\eta_k\gamma_k^r, \\ &    \Lambda_{N,6} \triangleq \left(1+\tfrac{1}{m}\right)\tfrac{\left(C+\eta_0C_f\right)^2}{2}\textstyle\sum_{k=0}^{N}\gamma_k^{r+1}. 
	\end{align*}
	From Proposition \ref{prop:error_bound_aIR-IG} (a) and (b), we have:
	\begin{align}\label{eqn:bounds_in_Lambdas_f}
		{f\left(\bar x_N\right)}-f(x^*) &\leq \left({\Lambda_{N,2}+\Lambda_{N,3}}\right)/{\Lambda_{N,1}}, \nonumber\\
		\phi\left(\bar{x}_N\right)&\leq  \left(\Lambda_{N,4}+\Lambda_{N,5}+\Lambda_{N,6}\right)/{\Lambda_{N,1}}.
	\end{align}
	Next, applying Lemma \ref{lem:harmonic_series_bound} and substituting $\{\gamma_k\}$ and $\{\eta_k\}$ by their update rules, we obtain:
	\begin{align*}
		\Lambda_{N,1} &=\textstyle\sum_{k=0}^{N}\tfrac{\gamma_0^r}{(k+1)^{0.5r}}\geq \tfrac{\gamma_0^r(N+1)^{1-0.5r}}{2(1-0.5r)}.\\  
		\Lambda_{N,2} &= \tfrac{2M^2(N+1)^{0.5(1-r)+b}}{\eta_0\gamma_0^{1-r}}.\\&\hspace{-0.7cm} 	\Lambda_{N,3} = \left(1+\tfrac{1}{m}\right)\tfrac{\left(C+\eta_0C_f\right)^2}{2}\textstyle\sum_{k=0}^{N}\tfrac{\gamma_0^{1+r}}{\eta_0(k+1)^{0.5(1+r)-b}}\\
		 \Lambda_{N,3}&\leq\tfrac{\left(m+1\right)\gamma_0^{1+r}\left(C+\eta_0C_f\right)^2(N+1)^{1-0.5(1+r)+b}}{2m\eta_0(1-0.5(1+r)+b)}.
		 	\end{align*}
	 \begin{align*}
		\Lambda_{N,4} &= \tfrac{2M^2(N+1)^{0.5(1-r)}}{\gamma_0^{1-r}}. \\
		\Lambda_{N,5} &=\textstyle\sum_{k=0}^{N}\tfrac{2M_f \eta_0\gamma_0^r}{(k+1)^{0.5r+b}}\leq  \tfrac{2M_f \eta_0\gamma_0^r (N+1)^{1-0.5r-b}}{1-0.5r-b}.\\
		\Lambda_{N,6} &= \left(1+\tfrac{1}{m}\right)\tfrac{\left(C+\eta_0C_f\right)^2}{2}\textstyle\sum_{k=0}^{N}\tfrac{\gamma_0^{r+1}}{(k+1)^{0.5(1+r)}}\\&\hspace{-0cm}\leq \tfrac{\left(m+1\right)\left(C+\eta_0C_f\right)^2\gamma_0^{r+1}(N+1)^{1-0.5(1+r)}}{2m(1-0.5(1+r))}.
	\end{align*} 
For these inequalities to hold, we need to ensure that conditions of Lemma \ref{lem:harmonic_series_bound} are met. Accordingly, we must have $0\leq 0.5r <1$, $0\leq 0.5(1+r)-b <1$, $0\leq 0.5r+b<1$, and $0 \leq 0.5(1+r)<1$. These relations hold because $0\leq r<1$ and $0<b<0.5$. Another set of conditions when applying Lemma \ref{lem:harmonic_series_bound} includes $N\geq \max\left\{2^{1/(1-0.5r)},2^{1/(1-0.5(1+r)+b)},2^{1/(1-0.5r-b)},\right.$ \\ $\left.  2^{1/(1-0.5(1+r))}\right\}-1$. Note that this condition is satisfied as a consequence of $N\geq 2^{\tfrac{2}{1-r}}-1$, $b>0$, and $0\leq r<1$. We conclude that all the necessary conditions for applying Lemma \ref{lem:harmonic_series_bound} and obtaining the aforementioned bounds for the terms $\Lambda_{N,i}$ are satisfied. To show that the inequalities \eqref{eqn:rate_f} and \eqref{eqn:rate_infeas}, it suffices to substitute the preceding bounds of  $\Lambda_{N,i}$, in the inequalities \eqref{eqn:bounds_in_Lambdas_f}.
\begin{align*}
		& \hspace{-0.05cm}{f\left(\bar x_N\right)}\hspace{-0.05cm}-\hspace{-0.08cm}f(x^*) \leq  \tfrac{2-r}{\gamma_0^r(N+1)^{1-0.5r}}\hspace{-0.05cm}\left(\hspace{-0.1cm}\tfrac{2M^2(N+1)^{0.5-0.5r+b}}{\eta_0\gamma_0^{1-r}}\right.\\
		&\qquad\quad\left.+\tfrac{\left(m+1\right)\gamma_0^{1+r}\left(C+\eta_0C_f\right)^2(N+1)^{0.5-0.5r+b}}{2m\eta_0(1-0.5(1+r)+b)}\right).
	\end{align*}
	Inequality \eqref{eqn:rate_f} is obtained by rearranging the terms in the preceding relation. Next, consider the following:
	\begin{align*}
		\phi\left(\bar{x}_N\right) \nonumber\leq&\tfrac{2-r}{\gamma_0^r(N+1)^{1-0.5r}}\left( \tfrac{2M_f \eta_0\gamma_0^r (N+1)^{1-0.5r-b}}{1-0.5r-b} \right.\\&\left.+ \tfrac{2M^2(N+1)^{0.5-0.5r}}{\gamma_0^{1-r}}\right.\\&\left. +\tfrac{\left(m+1\right)\left(C+\eta_0C_f\right)^2\gamma_0^{r+1}(N+1)^{0.5-0.5r}}{2m(1-0.5(1+r))} \right),\nonumber\\
		\leq& (2-r)\left(\tfrac{2M^2}{\gamma_0(N+1)^{0.5}} + \tfrac{2M_f \eta_0 }{(1-0.5r-b)(N+1)^{b}} \right.\nonumber\\
		&\qquad\qquad\left. +\tfrac{\left(m+1\right)\left(C+\eta_0C_f\right)^2\gamma_0}{2m(0.5-0.5r)(N+1)^{0.5}} \right).
		\end{align*}
		Taking into account $0<b<0.5$, 
%	\begin{align*}
%		&\phi\left(\bar{x}_N\right) \leq (2-r)\left(\tfrac{2M^2}{\gamma_0(N+1)^{b}} + \tfrac{2M_f \eta_0 \gamma_0^r}{(1-0.5r-b)(N+1)^{b}} \right.\nonumber\\
%		&\left. +\tfrac{\left(m+3\right)\left(C^2+\eta_0^2C_f^2\right)\gamma_0^{r+1}}{2m(1-0.5(1+r))(N+1)^{b}}  +\tfrac{\left(m+1  \right)CC_f\gamma_0^{1+r}\eta_{0} }{m\left(1-0.5(1+r)\right)(N+1)^{b}} \right).
%	\end{align*}
	equation \eqref{eqn:rate_infeas} is obtained by rearranging the terms in the preceding inequality.
\end{proof} 

\begin{remark} \normalfont The convergence rates derived in Theorem \ref{thm:rates} can be improved under stronger assumptions such as smoothness and strong convexity of the functions $f_i$. Indeed, this is a future direction of our study. We note that a preliminary version of this paper where a better rate has been derived under such assumptions is \cite{KaushikYousefian2020}. We have omitted such discussions due to the space limitation.
\end{remark}

%%%%%%%%%%%%%%%%%%%%

\section{Numerical Results} \label{sec:numerical_implementation}

%
%\begin{table}[b]
%	\centering
%	\begin{tabular}{|r|r|r|}
%		\hline
%		Initial point          & Gradient Descent            & Gauss Newton                \\ \hline
%		\multirow{5}{*}{(0,0)} & local convergence           & local convergence           \\
%		& u\_0: 0.000                 & u\_0: 0.000                 \\
%		& u: 0.000                    & u: 0.000                    \\
%		& f\textasciicircum{}*: 2.500 & f\textasciicircum{}*: 2.500 \\
%		& time: 12.224                & time: 12.306                \\ \hline
%		\multirow{5}{*}{(1,5)} & local convergence           & local convergence           \\
%		& u\_0: 1.000                 & u\_0: 1.000                 \\
%		& u: 5.000                    & u: 5.000                    \\
%		& f\textasciicircum{}*: 5.738 & f\textasciicircum{}*: 5.738 \\
%		& time: 13.692                & time: 13.724                \\ \hline
%		\multirow{5}{*}{4,-20} & local convergence           & local convergence           \\
%		& u\_0: 4.000                 & u\_0: 4.000                 \\
%		& u: -20.000                  & u: -20.000                  \\
%		& f\textasciicircum{}*: 5.964 & f\textasciicircum{}*: 5.964 \\
%		& time: 13.411                & time: 14.042                \\ \hline
%	\end{tabular}
%\end{table}

\begin{table*}[t]
	\renewcommand\thetable{1}
	\captionsetup{labelformat=empty}
	\centering
	\begin{tabular}{c|c|c}
		%		\hline
		$ {N} \backslash {n}$  & 50 & 100 \\
		%		\cmidrule{1-3}
		\hline
		100&\begin{minipage}{7.5cm}\centering\includegraphics[width=0.5\textwidth]{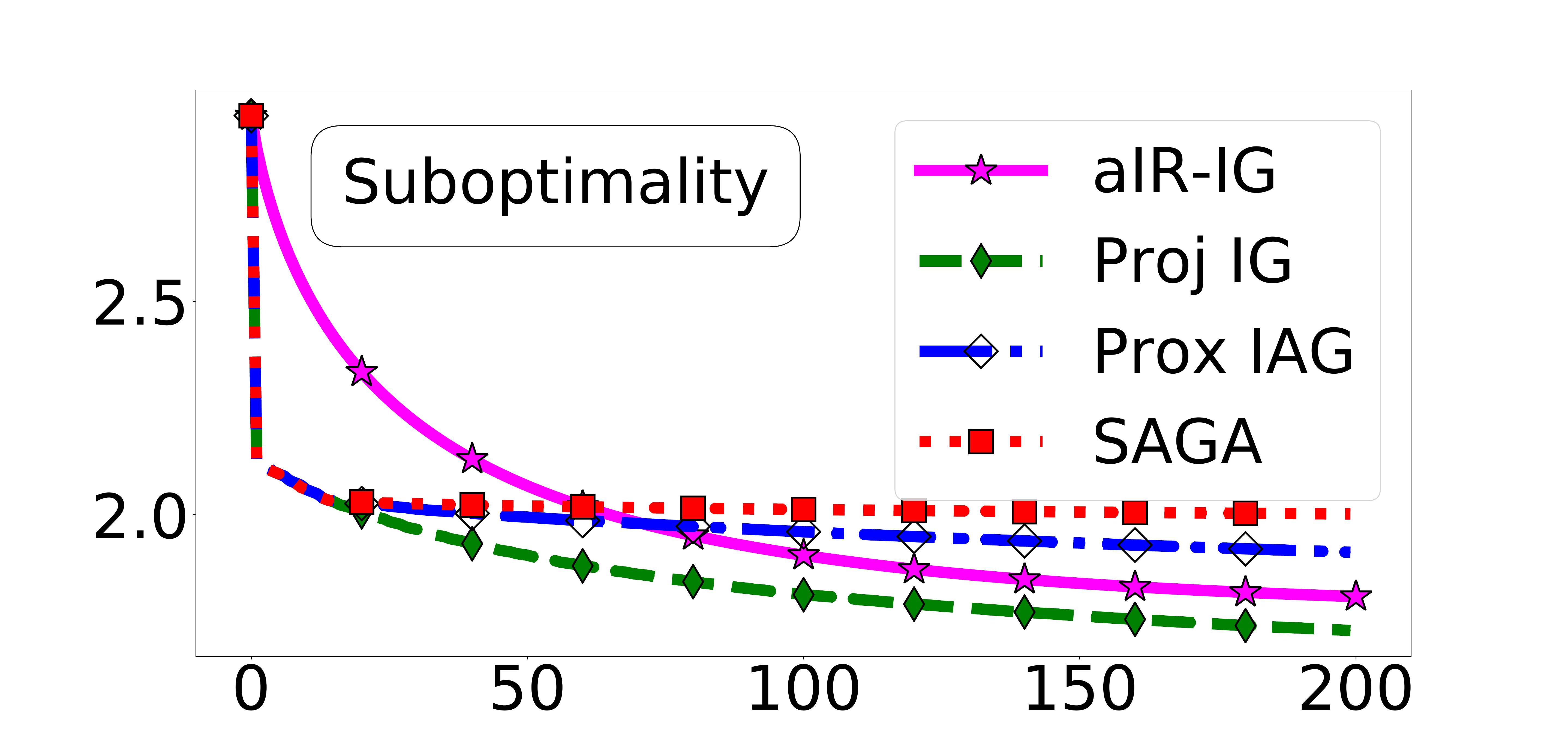}\includegraphics[width=0.5\textwidth]{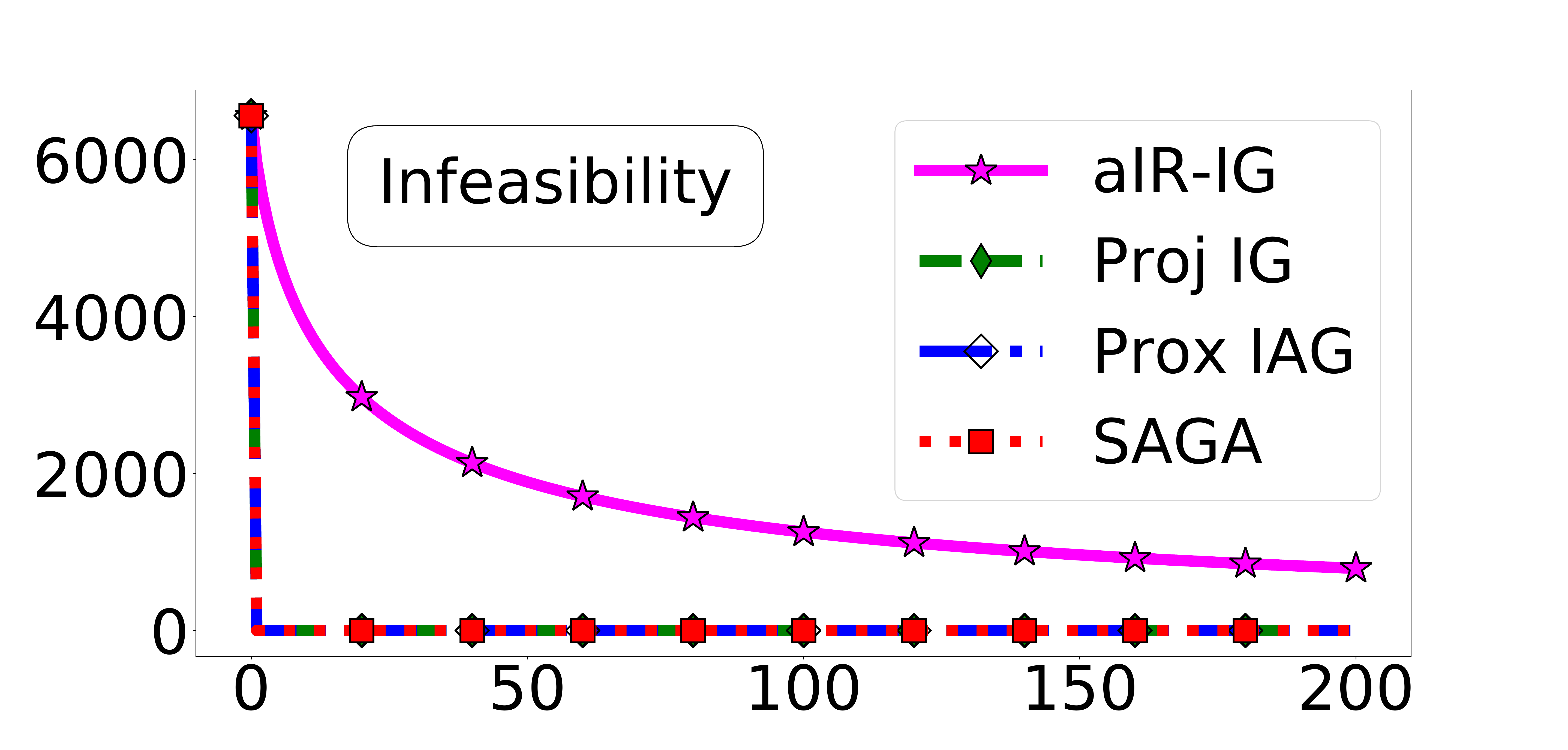}\end{minipage}
		&\begin{minipage}{7.5cm}\centering\includegraphics[width=0.5\textwidth]{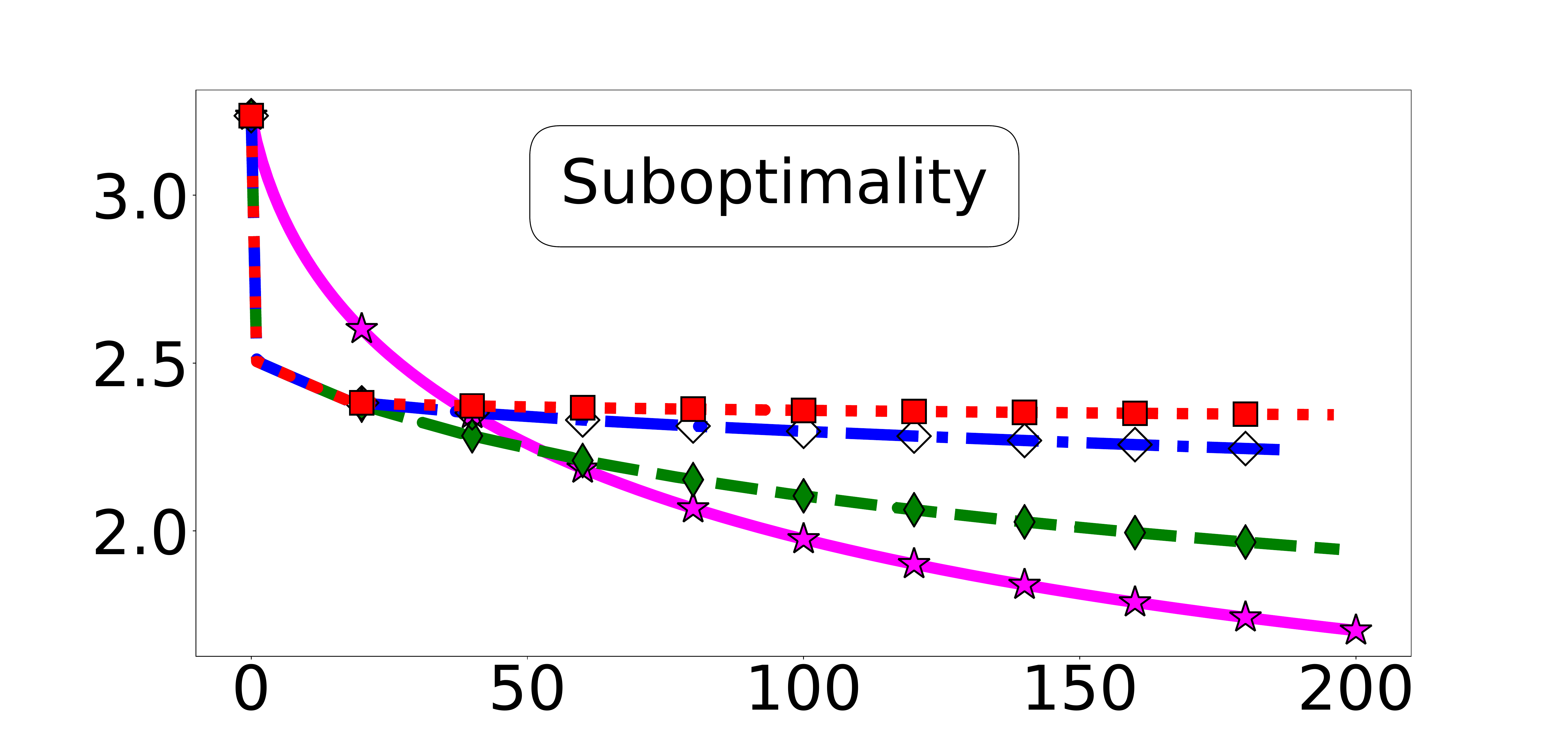}\includegraphics[width=0.5\textwidth]{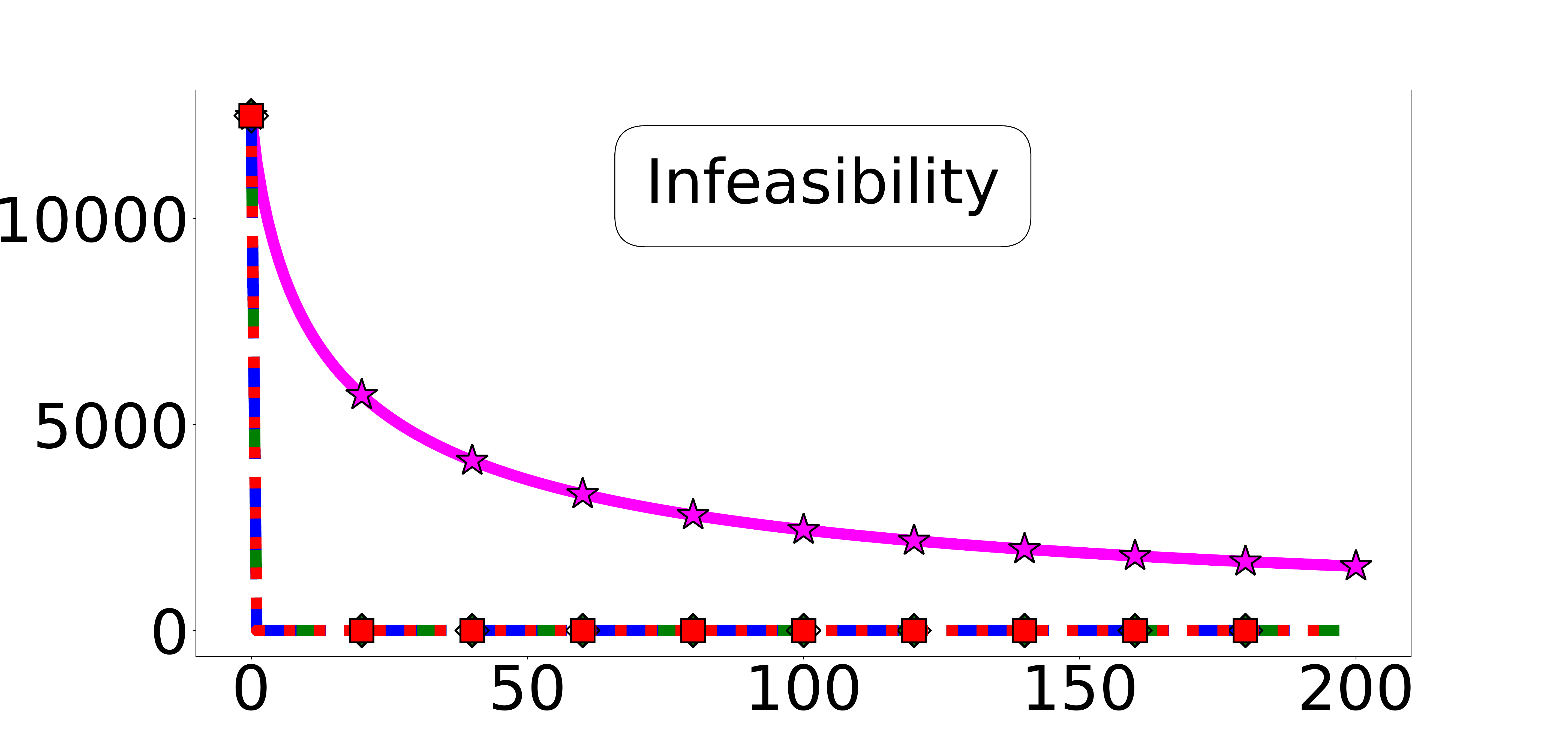}\end{minipage}
		\\\hline 200 &\begin{minipage}{7.5cm}\centering\includegraphics[width=0.5\textwidth]{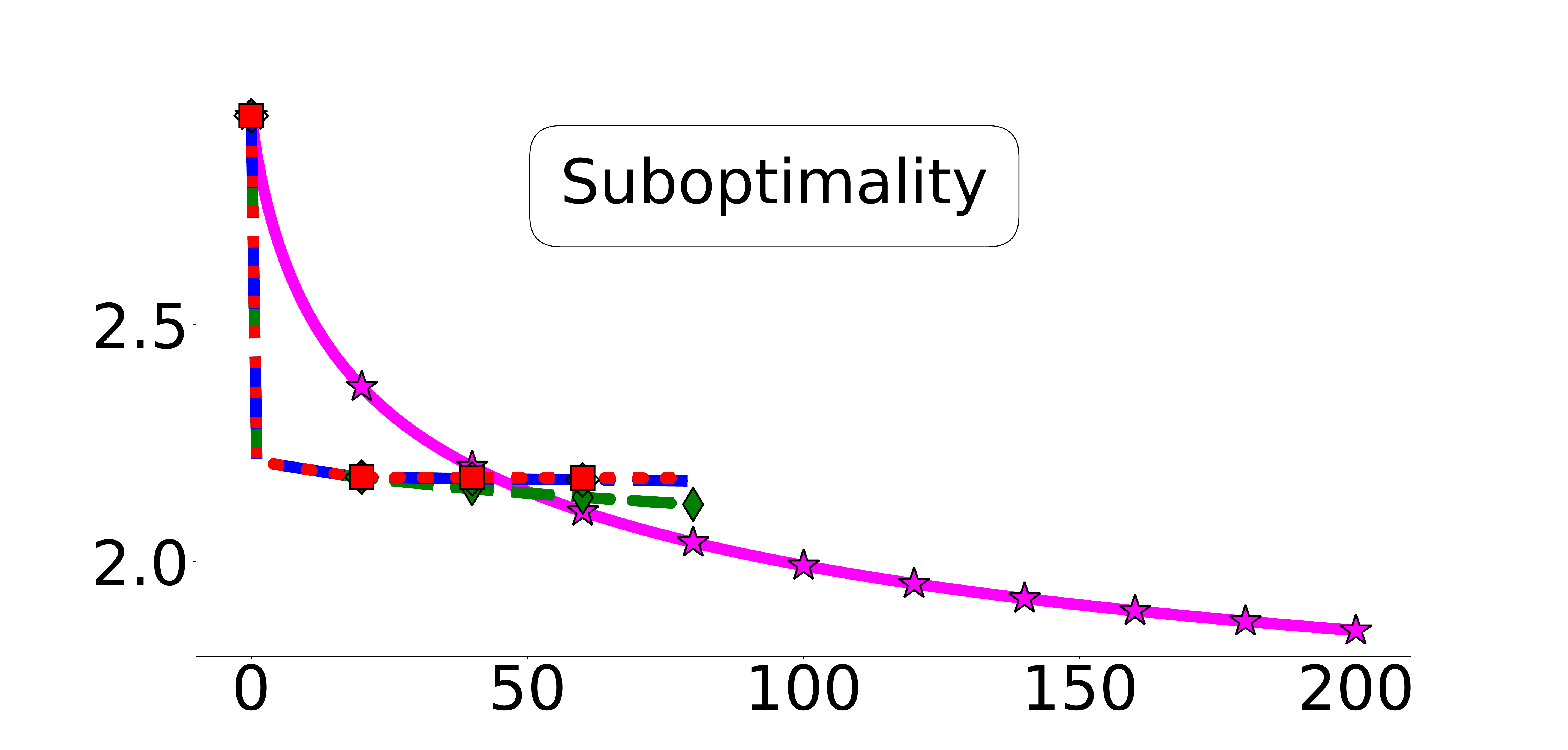}\includegraphics[width=0.5\textwidth]{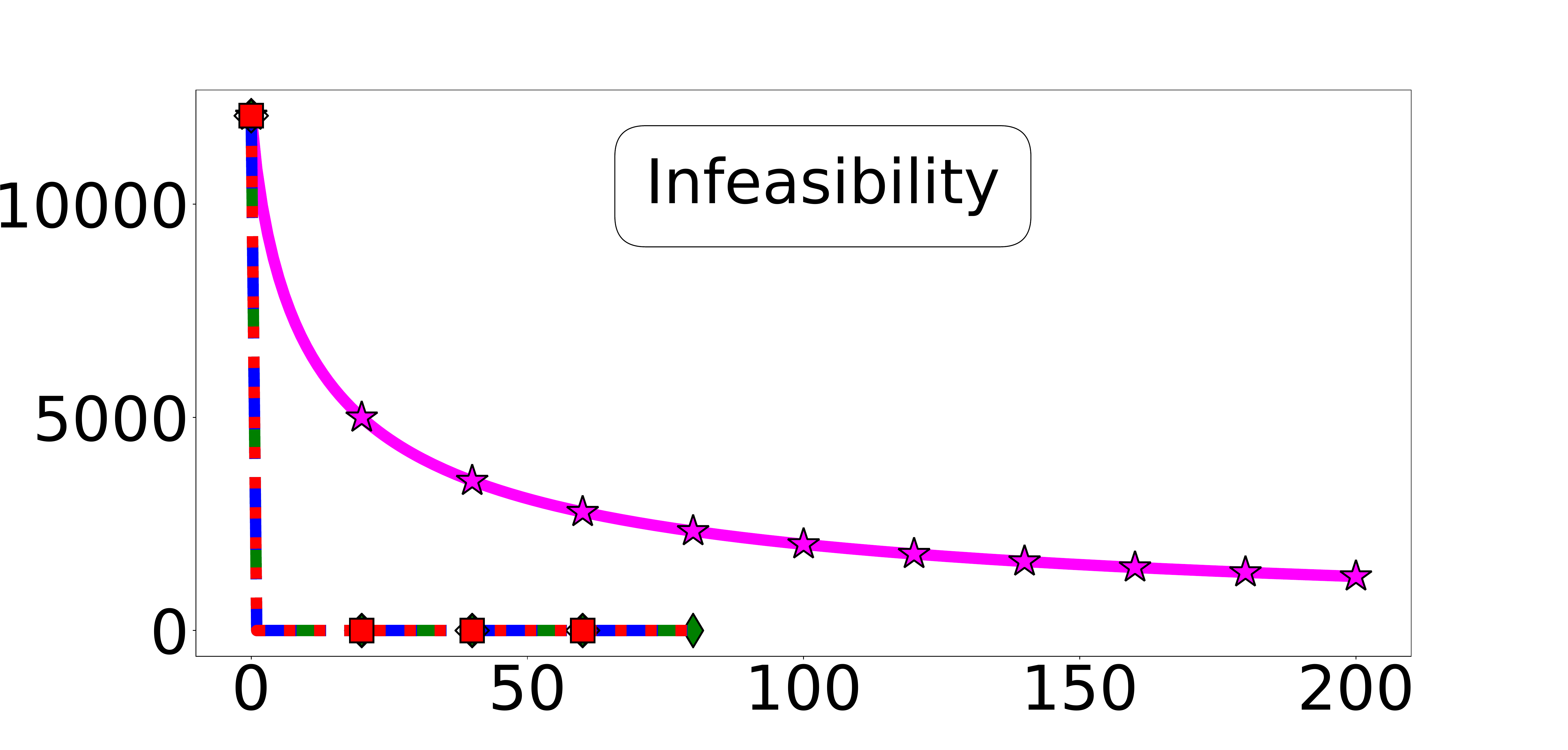}\end{minipage}
		&\begin{minipage}{7.5cm}\centering\includegraphics[width=0.5\textwidth]{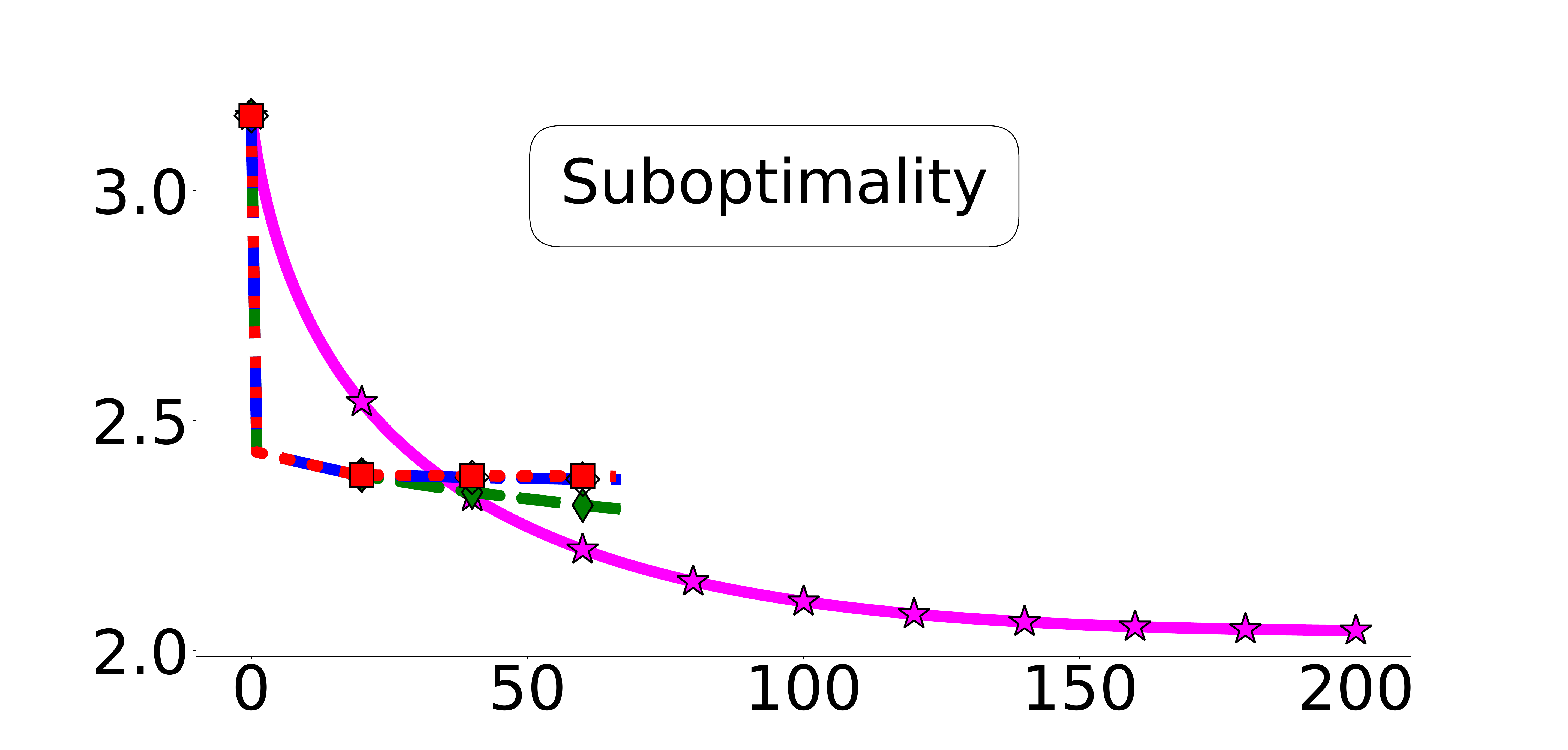}\includegraphics[width=0.5\textwidth]{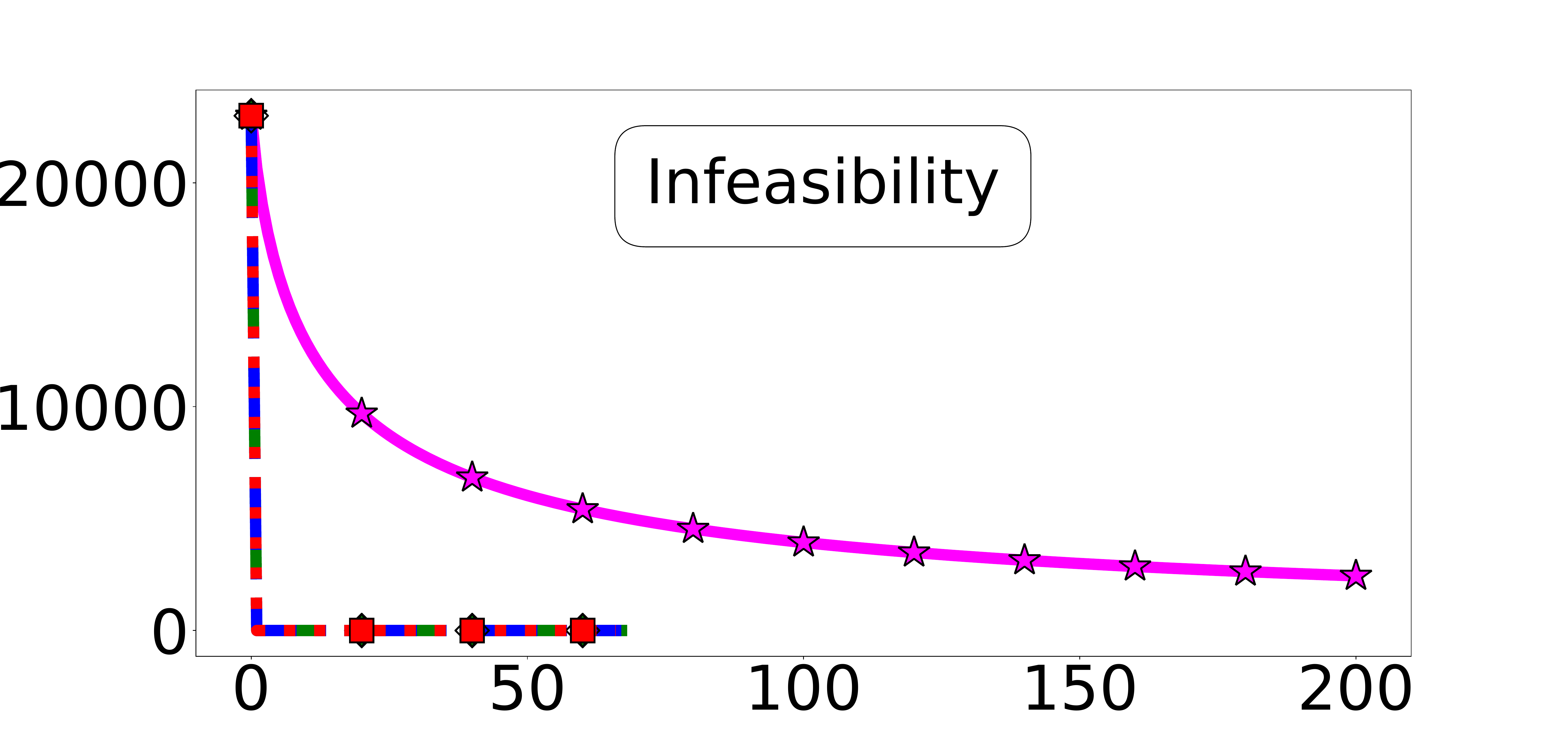}\end{minipage}
		\\\hline 500 &\begin{minipage}{7.5cm}\centering\includegraphics[width=0.5\textwidth]{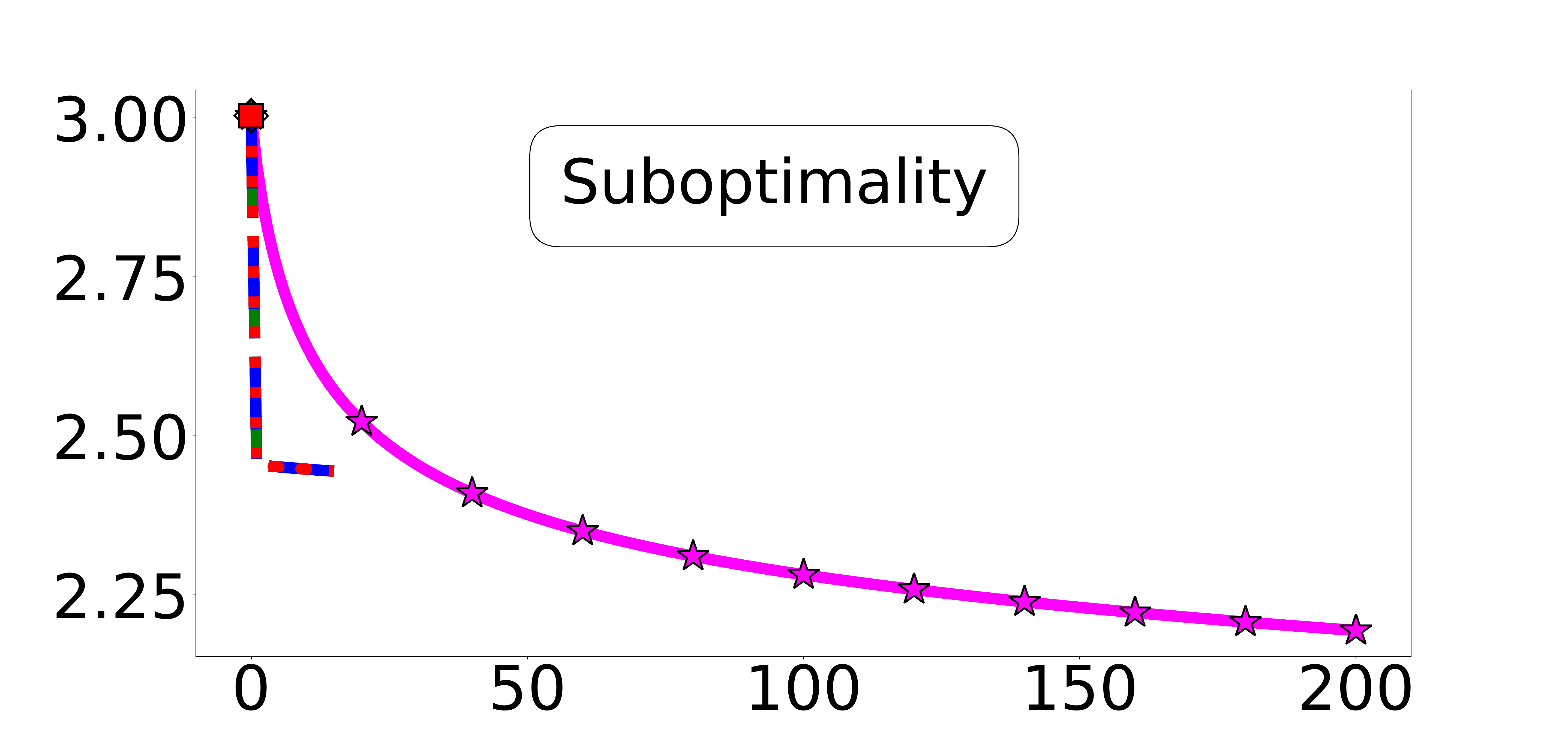}\includegraphics[width=0.5\textwidth]{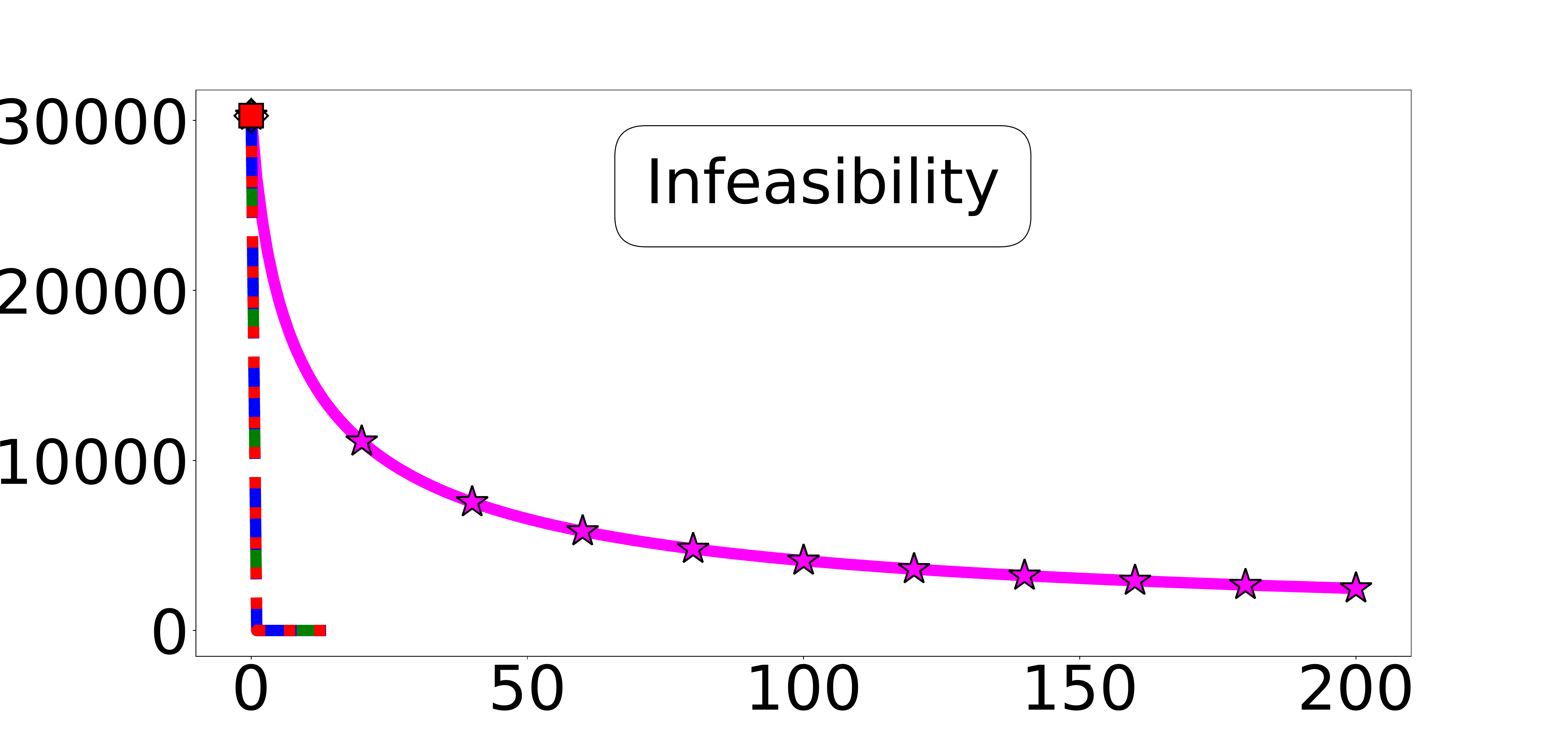}\end{minipage}
		&\begin{minipage}{7.5cm}\centering\includegraphics[width=0.5\textwidth]{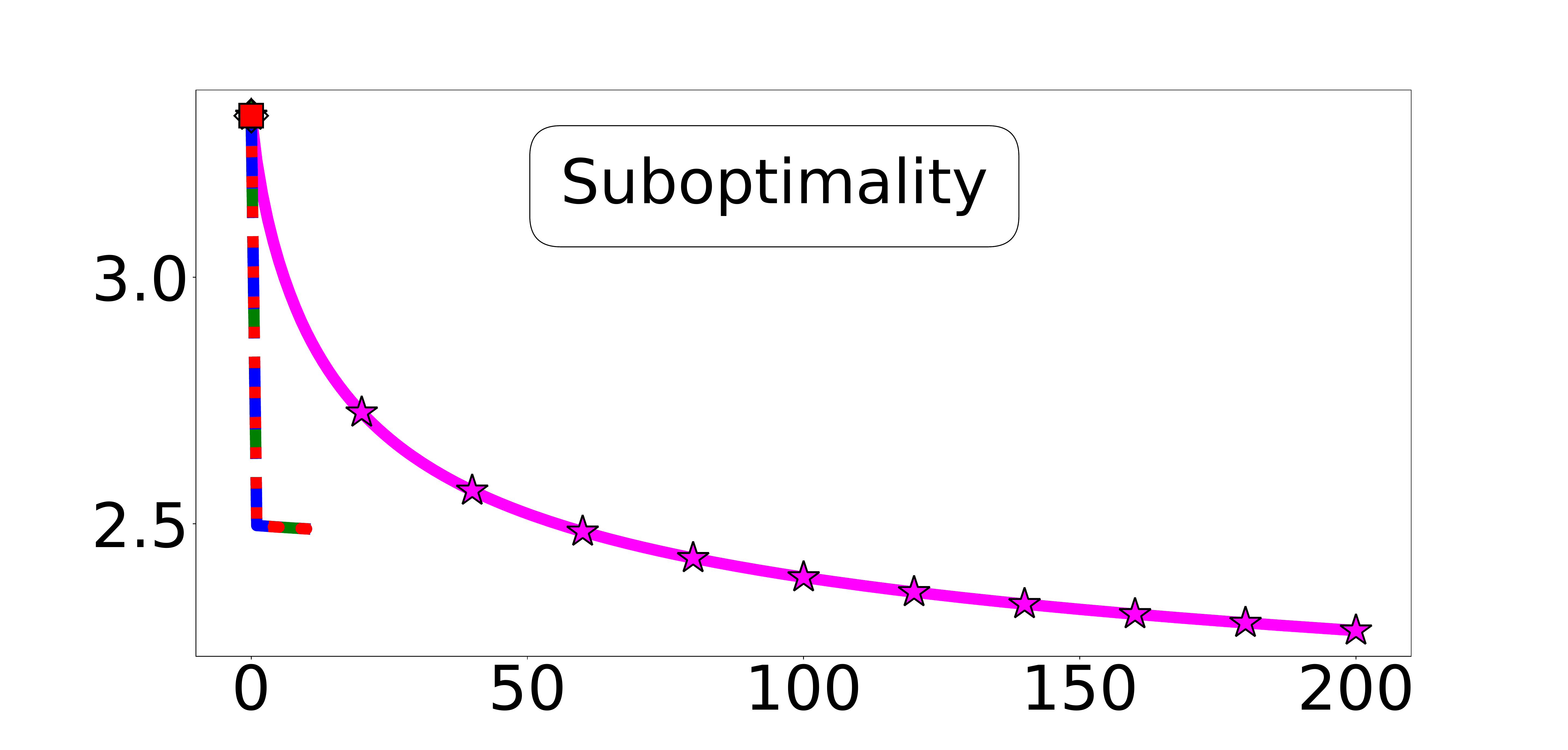}\includegraphics[width=0.5\textwidth]{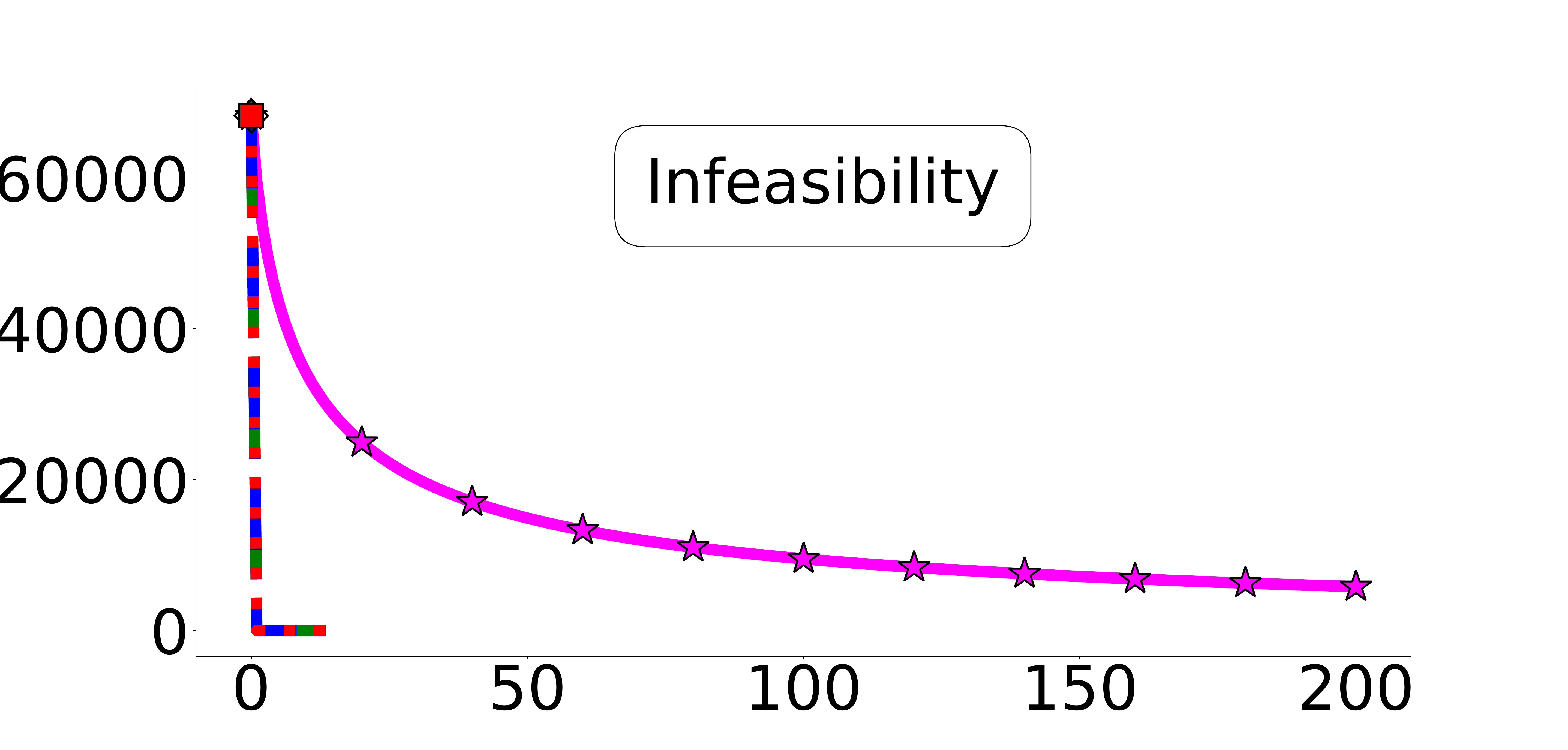}\end{minipage}\\\hline
	\end{tabular}
	\caption{\label{tab:function_value}{Figure 1}: Comparison of suboptimality and infeasibility of Algorithm \ref{alg:IR-IG_avg}, projected IG, proximal IAG, and SAGA  \hdk{over} time.}
\end{table*}

%\begin{table}
%	\renewcommand\thetable{1}
%	\captionsetup{labelformat=empty}
%	\centering
%	\begin{tabular}{c|ccc}
%		%\hline
%		$ {p} \backslash {n}$  & 800 & 10000 \\
%		\hline  \noalign{\vskip 0.1mm} 
%		\centering 800&\raisebox{-0.4\totalheight}{\includegraphics[width=0.4\textwidth]{f31.pdf}}% \includegraphics[width=0.4\textwidth]{f31.pdf} 
%		&\raisebox{-0.4\totalheight}{\includegraphics[width=0.4\textwidth]{1f32.pdf} }
%		%		&\includegraphics[width=0.2\textwidth]{f3233.pdf}  
%		\\1000&\raisebox{-0.4\totalheight}{\includegraphics[width=0.4\textwidth]{f21.pdf} }
%		&\raisebox{-0.4\totalheight}{\includegraphics[width=0.4\textwidth]{1f22.pdf} }
%		%&\includegraphics[width=0.2\textwidth]{f3233.pdf} 
%		\\5000&\raisebox{-0.4\totalheight}{\includegraphics[width=0.4\textwidth]{1f11.pdf} }
%		&\raisebox{-0.4\totalheight}{\includegraphics[width=0.4\textwidth]{f12.pdf} }
%		\\7000&\raisebox{-0.4\totalheight}{\includegraphics[width=0.4\textwidth]{f41.pdf} }
%		&\raisebox{-0.4\totalheight}{\includegraphics[width=0.4\textwidth]{f42.pdf} }
%		%\hline
%	\end{tabular}
%	\caption{\\ \label{tab:collection_of_figs}{Figure 1}: Comparison of the performance metric $\text{log}_{10}\left(f(x)+\|Ax-b\|^2\right)$ for Algorithm \ref{alg:update_rule_IR-IG}, projected IG, IAG, SAG, and SAGA  vs. time\vspace{-0.5cm}}
%\end{table}

In this section, we present the simulations for the proposed algorithm on  a distributed soft-margin support vector machine (SVM). We compare the performance of aIR-IG with the state-of-the-art IG schemes, \hdk{including the} projected IG, proximal IAG, and SAGA. The schemes are compared in terms of CPU time. For \hdk{these} numerical experiments, we  use the soft-margin formulation of SVM, as follows: 
\begin{align}\label{prob:numerics}
	    & \underset{w,b,\hdk{z}}{\text{minimize}} \qquad  \tfrac{1}{2}\|w \|^2 + \tfrac{1}{\lambda}\textstyle\sum_{i =1}^N z_i \\
	& \text{subject to}\quad   v_i(\hdk{w^T}{u}_i + b)   \geq 1 - z_i \quad\ \text{ for } \ i = 1, \dots, N.\nonumber\\
	& \qquad \qquad\quad \ \ \qquad \qquad  z_i \geq 0  \quad  \qquad  \ \text{ for } \ i = 1, \dots, N\nonumber.
\end{align}
Here,  $ \left( u_1, v_1 \right), \left( u_2, v_2 \right), \dots, \left( u_N, v_N \right)$ denote the  dataset  such that  $ u\in \mathbb{R}^n $ and $ v\in\{-1,+1\}$. The goal here is to find a classifier given by $ w^Tu + b $  to separate the two  classes of  $\ v := +1 \text{ and }  v := -1$, whereas $w \in \mathbb{R}^n  \text{ and } b \in \mathbb{R}.$ For a distributed implementation, we define the objective for agent $i  \in \{1, \dots, m\} $ as follows: 
\begin{align*}
f_i(w,z_i) =  \textstyle\sum_{j = \tfrac{N\times (i-1)}{m} +1}^{\hdk{\tfrac{N}{m}\times i}} \tfrac{1}{2N}\| w \|^2 + \tfrac{1}{\lambda}z_j.
\end{align*}
Recall that Algorithm \ref{alg:IR-IG_avg} does not require any projection onto the feasible set. However, in other schemes including IG, proximal IAG,  and SAGA a projection (more generally a proximal step) is needed. For convenience, define $ x \triangleq (w^T,b,z^T)^T$. Now for evaluating \hdk{the} projection of vector $x_1 \triangleq (w_1^T,b_1,z_1^T)^T$, we solve the following optimization problem: 
\begin{align}\label{prob:projection}
 \underset{w,b,\hdk{z}}{\text{min}}  \left\{ \tfrac{\|x - x_1 \|^2}{2} \bigg\vert
    v_i(\hdk{w^T}{u}_i + b)   \geq 1 - z_i,   z_i \geq 0  \ \forall     i \in [N]\right\}
\end{align}

{\bf Simulation Platform. }  All the  simulations are implemented using Python on a computer with 16 GB RAM. We use \hdk{the} Gurobi-Python interface to solve  projection problem \eqref{prob:projection}.

{\bf Set up.} The simulations were performed for $m=20$ agents,  $\lambda = 10$, $\gamma_{{0}} = \eta_{{0}} = 1$, and $b=0.25$. For this experiment, time was fixed to $200$ seconds and the performance of each scheme is recorded.   Figure \ref{tab:function_value} shows the performance of Algorithm \ref{alg:IR-IG_avg}, projected IG, proximal IAG, and SAGA for the different choices of dimensionality $n$ and the total number of samples $N$. Performance is recorded in terms of suboptimality and infeasibility where suboptimality is  $\tfrac{1}{2}\|w \|^2 + \tfrac{1}{\lambda}\textstyle\sum_{i =1}^N z_i$ and infeasibility is the violation of constraints of problem \eqref{prob:numerics}. Suboptimality is shown in a logarithmic scale in Figure  \ref{tab:function_value}. %for better visualization. %Note that this comparison is done with respect to time. 

{\bf Insights. }  With increasing the dimension and the number of samples, the projection evaluations {take} longer and consequently, the performance of the projected variant of the aforementioned IG schemes is deteriorated. This is the case in particular when $N=500$. \hdk{Note that the other schemes, namely Proj IG, Prox IAG, and SAGA do not show any update for $N =$ 200 and 500 after about 70 seconds and 20 seconds, respectively. This is because of the interruption in their last update due to reaching the time limit of 200 seconds.}

\section{Concluding remarks}We consider the problem of minimizing the \hdk{finite sum with separable (agent-wise) nonlinear inequality and linear equality and inequality constraints.} Our work is motivated by the computational challenges in \hdk{the} projected incremental gradient schemes under the presence of hard-to-project constraints. We develop an averaged iteratively regularized incremental gradient scheme where we employ a novel regularization-based relaxation technique. The proposed algorithm is designed in a way that it does not require a hard-to-project computation. We establish the rates of convergence for the objective function value and the infeasibility of the generated iterates.  We compare  the proposed scheme with  the state-of-the-art incremental gradient schemes including projected IG, proximal IAG, and SAGA. We observe that the proposed scheme outperforms the projected schemes as the number of samples or the dimension of the solution space increases.


%%%%%%%%%%%%%%%%%%%%%%%%%%%%%%%%%%%%%%%%%%%%%%%%%%%%%%%%%%%%%%%%%%%%%%%%%%%%%%%%

%%%%%%%%%%%%%%%%%%%%%%%%%%%%%%%%%%%%%%%%%%%%%%%%%%%%%%%%%%%%%%%%%%%%%%%%%%%%%%%%

%%%%%%%%%%%%%%%%%%%%%%%%%%%%%%%%%%%%%%%%%%%%%%%%%%%%%%%%%%%%%%%%%%%%%%%%%%%%%%%%

%\section*{\hdk{Acknowledgment}}

%%%%%%%%%%%%%%%%%%%%%%%%%%%%%%%%%%%%%%%%%%%%%%%%%%%%%%%%%%%%%%%%%%%%%%%%%%%%%%%%

%
%\begin{thebibliography}{99}
%\bibitem{c1} G. O. 
%\bibitem{c2} W. K. Chen, Linear Networks and S
%\end{thebibliography}

\bibliographystyle{IEEEtran}
\bibliography{ref_fy_hdk,references_hdk}

\end{document}